\documentclass[10pt,a4paper]{amsart}
\usepackage{amsmath, graphicx, subfigure, epsfig,amssymb}
\usepackage{xcolor}
\numberwithin{equation}{section}

\newcommand{\mc}{\nu}
\newcommand{\e}{\epsilon}

\newcommand{\de}{\delta}
\newcommand{\br}{\mathbb{R}}
\newcommand{\ik}{\varphi}
\newcommand{\pa}{\partial}
\newcommand{\bt}{\beta}
\newcommand{\al}{\alpha}
\newcommand{\la}{\lambda}
\newcommand{\om}{\omega}

\newcommand{\coi}{C_0^{\infty}}
\newcommand{\ioi}{\int_0^{\infty}}
\newcommand{\oy}{\overline Y}
\newcommand{\xh}{x_\e}
\newcommand{\be}{\begin{equation}}
\newcommand{\ee}{\end{equation}}

\newcommand{\nal}{n}
\newcommand{\us}{\mathcal U}
\newcommand{\vs}{\mathcal V}
\newcommand{\ir}{\mathcal C}

\newcommand{\s}{\mathcal S}
\newcommand{\T}{\mathcal T}
\newcommand{\R}{\mathcal R}

\theoremstyle{theorem}
\newtheorem{theorem}{Theorem}
\newtheorem*{remark}{Remark}
\newtheorem{lemma}{Lemma}

\newtheorem{definition}{Definition}

\begin{document}

\title[Analysis of resolution]{Resolution analysis of inverting the generalized Radon transform from discrete data in $\mathbb R^3$}
\author[A Katsevich]{Alexander Katsevich$^1$}
\thanks{$^1$Department of Mathematics, University of Central Florida, Orlando, FL 32816.\\ 
This work was supported in part by NSF grant DMS-1615124.}

\begin{abstract}
A number of practically important imaging problems involve inverting the generalized Radon transform (GRT) $\R$ of a function $f$ in $\br^3$. On the other hand, not much is known about the spatial resolution of the reconstruction from discretized data. In this paper we study how accurately and with what resolution the {\it singularities} of $f$ are reconstructed. The GRT integrates over a fairly general family of surfaces $\s_y$ in $\br^3$. Here $y$ is the parameter in the data space, which runs over an open set $\vs\subset\br^3$. Assume that the data $g(y)=(\R f)(y)$ are known on a regular grid $y_j$ with step-sizes $O(\e)$ along each axis, and suppose $\s=\text{singsupp}(f)$ is a piecewise smooth surface. Let $f_\e$ denote the result of reconstruction from the discrete data. We obtain explicitly the leading singular behavior of $f_\e$ in an $O(\e)$-neighborhood of a generic point $x_0\in\s$, where $f$ has a jump discontinuity. We also prove that under some generic conditions on $\s$ (which include, e.g. a restriction on the order of tangency of $\s_y$ and $\s$), the singularities of $f$ do not lead to non-local artifacts. For both computations, a connection with the uniform distribution theory turns out to be important. Finally, we present a numerical experiment, which demonstrates a good match between the theoretically predicted behavior and actual reconstruction.
\end{abstract}
\maketitle

\section{Introduction}\label{sec_intro}

A large number of practically important imaging problems involve inversion of the generalized Radon transform (GRT), i.e. recovering an unknown function $f$ from its integrals over a family of surfaces. The reconstruction may involve finding $f$ itself, or finding $f$ modulo smoother terms. Most of the times, the surfaces are not planes. Below is a list of some of the most common integral transforms with some of the most prominent examples of their use.   
\begin{enumerate}
%\item The classical Radon transform integrates over lines in $\br^2$, and over planes and hyperplanes -- in higher dimensions \cite{nat3}. Such transforms arise in various applications of CT, e.g. medical imaging, security scanning, non-destructive testing, and many others.
%\item Generalized Radon transform of dynamic objects (integration over lines or surfaces determined by the deformation of the object). This transform arises, for example, in medical imaging when scanning the cardiac region of the patient \cite{tcst12}. 
\item Integration over spheres. Applications include ultrasound imaging (or, SONAR) (see \cite{qrs11} and references therein) as well as thermoacoustic and photoacoustic tomography \cite{kk15, wa15}. 
\item Integration over ellipses. This transform arises in linearized seismic imaging with a common offset between the sources and receivers \cite{gkqr18}. 
\item Integration over cones arises in Compton camera imaging. Applications are single-scattering optical tomography, Compton camera medical imaging, and homeland security (see \cite{tkk18} for a recent review).
\end{enumerate}

In all of the above cases, one collects a discrete data set and reconstructs $f$ using a numerical algorithm. Frequently, reconstruction is achieved by applying a linear inversion formula (as opposed to a non-linear reconstruction algorithm based on fidelity functional minimization). In all of the above examples it is of fundamental importance to know the resolution of the method as a function of (1) the data sampling rate, and (2) specific implementation of the inversion formula that is used. Despite the significance of this problem, not much is known about the resolution of reconstruction from discrete data. The main reason for this is that the classical sampling theory, which addresses such problems, can be applied only in a few simplest settings of the classical Radon transform (CRT) \cite{nat3}. 
%For none of the above transforms there exists a comprehensive analysis of resolution based on the classical sampling theory. 
The known results are quite scarce, and they are of semi-qualitative nature (see e.g. pp. 784--786 in \cite{cb15}). 
Very recently, a more flexible approach to sampling based on semiclassical analysis was proposed in \cite{stef18}. Let $A$ be a Fourier Integral Operator (FIO). The idea of \cite{stef18} is to determine how the data $Af$ should be sampled to allow for accurate 
%(i.e., up to the terms of order $O(\e^\infty)$) 
interpolation of its values on a lattice provided that $f$ is semiclassically bandlimited. If the sampling condition is violated, then reconstruction from the discrete values of $Af$ (i.e., applying a parametrix $A^{-1}$ to the interpolated $Af$) leads to aliasing artifacts, which are also analyzed in \cite{stef18}.

An alternative approach to the analysis of resolution was proposed recently in \cite{kat_2017, kat19a}. The idea is to investigate how accurately and with what resolution the {\it singularities} of $f$ are reconstructed. For some of the above problems there is no exact inversion formula, and inversion modulo smoother terms is the most one can hope for. In such cases, spatial resolution of the recovery of singularities is all one needs. Note that in this paper both 
$f$ and $g=\R f$ are assumed to have singularities in the sense of a conventional, classical wavefront set (see e.g. \cite{hor}). 
In contrast, the main assumption in \cite{stef18} is that $f$ and, consequently, the data $Af$ have only semiclassical singularities (see e.g. \cite{zwor12}). It is possible to apply the approach of \cite{stef18} to the analysis of classical singularities, but this would require summing a series over  ``folded'' frequencies in the Fourier domain, which is complicated. 

In \cite{kat_2017, kat19a} the author considers the inversion of the CRT of $f$ in $\br^2$ and $\br^3$.  The parametrization of the data is standard, i.e. in terms of the affine and angular variables. Suppose the step-sizes along the angular and affine variables are $O(\e)$. Let $f_\e$ denote the result of reconstruction from the discrete data. The author picks a point $x_0$, where $f$ has a jump singularity, and obtains explicitly the leading singular behavior of $f_\e$ in an $O(\e)$-neighborhood of $x_0$ as $\e\to0$. The obtained behavior, which we call {\it edge response}, provides the desired resolution of the reconstruction algorithm. It is shown also that convex parts of the singular support of $f$ do not create non-local artifacts. The case when $f$ changes during the scan (so-called, dynamic CT) is considered in the 2D setting as well \cite{kat_2017}. 

In this paper we generalize the approach of \cite{kat_2017, kat19a}. The reconstruction problem is now formulated in terms of the GRT $\R$, which integrates $f$ over a fairly general family of surfaces $\s_y$ in $\br^3$. Here $\text{supp}(f)\subset\us$, where $\us\subset\br^3$ is an open set, and $y$ is the parameter in the data space. For the problem to be well-determined, we assume that $y$ runs over an open set $\vs\subset\br^3$. As is seen, our setting is fairly general and covers all the problems mentioned above. The GRT in this paper is very close to that considered by Beylkin in \cite{belk}, only the parametrization of the surfaces $\s_y$ is slightly different. This gives us more flexibility to connect our results with practical applications, where GRTs arise.

Assume that the data $g=\R f$ are known on a regular grid $y_j$ with step-sizes $O(\e)$ along each axis. Suppose $\s=\text{singsupp}(f)$ is a piecewise smooth surface. Similarly to \cite{kat_2017, kat19a}, we obtain explicitly the leading singular behavior of $f_\e$ in an $O(\e)$-neighborhood of a generic point $x_0\in\s$, where $f$ has a jump discontinuity. We also prove that under some generic conditions on $\s$ (which include, e.g. a restriction on the order of tangency of $\s_y$ and $\s$), the singularities of $f$ do not lead to non-local artifacts. For both computations, a connection with the uniform distribution theory \cite{KN_06} turns out to be important. It is possible that violation of the imposed conditions leads to artifacts. Analysis of such artifacts and analysis of more general surfaces $\s$ will be the subject of future research. 
%Comprehensive analysis of all such cases is beyond the scope of this paper. 

The reconstruction formula $g\to \check f$, which contains a suitably adapted adjoint $\R^*$, is one specific example of an FIO. Here $\check f$ is such that $\check f-f$ is smoother than $f$. Thus, the reconstruction algorithm can be viewed as an application of an FIO to discrete data $(\R f)(y_j)$. A number of methods for computing the action of FIOs on discrete data have been proposed, see e.g. \cite{cdy07, cdy09, ahw12, yang15} and references therein. To the best of the author's knowledge, proposed is the first method to compute the {\it resolution} of the reconstruction obtained by applying an FIO to discrete data that comes from an image with classical singularities. Extension of the method to more general FIOs will also be the subject of future work. 

%\textcolor{red}{In summary, even though both this paper and \cite{stef18} investigate the resolution of reconstruction, they are quite different in underlying assumptions, methods, and conclusions. The approaches in \cite{stef18} and in this paper are complementary. They tackle the analysis of resolution in different ways and lead to non-overlapping, complementary results. }

The paper is organized as follows. In Section~\ref{prelim} we define the GRT $\R$ via an incidence relation $\ir\subset\us\times\vs$, list the properties of the function $\Phi(x,y)$ that defines the incidence relation, define generic points, and specify the continuous and discrete inversion formulas that are used in the analysis. The main result is formulated in Section~\ref{local-grt}, which also contains the beginning of the proof. The entire proof spans Sections~\ref{local-grt}--\ref{remote}. In Section~\ref{local-grt} we obtain the behavior of $g$ near its singular support, which generalizes one of the results of \cite{rz2, rz1} from the CRT to the GRT. The behavior of the interpolated data $g_\e$ near $\text{singsupp}(g)$ is obtained in Section~\ref{locint}. The contribution of the leading singular term to the edge response at a generic point $x_0\in \s$ is computed in Section~\ref{piece-one}. In Section~\ref{lot} we show that lower order terms do not contribute to the edge reponse. In Section~\ref{remote} we prove that, under some assumptions, remote singularities do not contribute to the edge response as well. Results of a numerical experiment, which show a good match between the theoretically predicted behavior and actual reconstruction, are in Section~\ref{numerix}. 

\section{Preliminary construction}\label{prelim}

Let $\us,\vs\subset\br^3$ be two open connected sets, where $\us$ is the image domain, and $\vs$ is the data domain. Each $y\in\vs$ determines a smooth surface $\s_y\subset \us$. Let $\ir$ be the corresponding incidence relation $\ir\in\us\times\vs$, which is defined in terms of a smooth function $\Phi(x,y)\in C^\infty(\us\times\vs)$:
\be\label{incidence}
\ir:=\{(x,y)\in \us\times\vs:\, \Phi(x,y)=0\}.
\ee
Another way to state \eqref{incidence} is that $x\in \s_y$ if and only if $\Phi(x,y)=0$. Define the submanifold:
\be\label{T-def}
\T_x:=\{y\in \vs:\, \Phi(x,y)=0\},\ x\in\us.
\ee
Thus, $\T_x$ is the collection of all $y\in\vs$ such that $\s_y$ contains $x$. The main assumptions about $\Phi$ are as follows (``DF'' stands for Defining Function):
\begin{itemize}
\item[DF1.] $\Phi$ is real-valued and non-degenerate, i.e.
\be\label{nondeg}
\Phi'_x(x,y)\not=0,\ \Phi'_y(x,y)\not=0,\ (x,y)\in \ir;
\ee
\item[DF2.] For each $x\in\us$, the map $\T_x\to S^2$ defined by $y\to \pm\Phi_x'(x,y)/|\Phi_x'(x,y)|$, $y\in\T_x$, is surjective; 
\item[DF3.] For each $y\in\vs$, the vectors $\Phi'_y(x,y)$ and $\Phi'_y(z,y)$ are not parallel whenever $x,z\in\s_y$, $x\not=z$; and
\item[DF4.] The mixed Hessian of $\Phi$ is non-degenerate, 
\be\label{bolker}
\text{det}\left(\frac{\pa^2\Phi(x,y)}{\pa x_i\pa y_j}\right)\not=0,\ (x,y)\in\ir,
\ee
where $\pa/\pa x_i$, $i=1,2$, and $\pa/\pa y_j$, $j=1,2$, are basis vectors in the tangent spaces to the submanifolds $\s_{y_0}$ and $\T_{x_0}$ at $x_0$ and $y_0$, respectively.
\end{itemize}
\noindent
Condition DF2 means that the tomographic data are complete, i.e. any singularity is visible. Condition DF3 says that there are no conjugate points. Condition DF4 is a local version of the Bolker condition. Conditions DF3 and DF4 imply the (global) Bolker condition. Conditions DF1--DF4 are analogous to Conditions (I)-(IV) in \cite{belk}. Conditions DF1 and DF4 combined are equivalent to the condition (cf. eq. (4.23), \cite{trev2}, p. 335) that at every point $(x,y)\in\ir$:
\be\label{big-det}
\text{det}\begin{pmatrix} \Phi''_{xy} &  (\Phi'_x)^T\\ \Phi'_y & 0 \end{pmatrix}\not=0.
\ee

In the paper we consider functions, which can be represented as a finite sum
\begin{equation}\label{f_def}
f(x)=\sum_j \chi_{D_j} f_j(x),
\end{equation}
where $\chi_{D_j}$ is the characteristic function of the domain $D_j\subset \us$. For each $j$:
\begin{enumerate}
\item $D_j$ is bounded,  
\item The boundary of $D_j$ is piecewise $C^{\infty}$,  
\item $f_j$ is $C^\infty$ in a domain containing the closure of $D_j$.
\end{enumerate}
Denote $\s:=\cup_j \pa D_j$. By construction, $\text{singsupp}(f)\subset \s$.

The GRT of $f$ is given by:
\be\label{grt_1}
g(y)=(\R f)(y):=\int _{S_y} b(x,y) f(x)dx,\ y\in\vs,
\ee
where the weight $b$ is smooth (i.e., $C^\infty$) and non-vanishing, and $dx$ is the area element on $\s_y$.
%\be\label{sy}
%S_y:=\{x\in D:\, \Phi(x,y)=0\},
%\ee
The discrete data are given by
\be\label{data}
g(\e j),\ j\in r+\mathbb Z^3,
\ee
for some $r\in\br^3$. 

Even though \eqref{data} assumes that the stepsize along each data axis equals $\e$, this is a non-restrictive assumption. Indeed, consider a smooth diffeomorphism $\psi$: $\vs\to\tilde\vs$ for some open $\tilde\vs\subset\br^2$, so that $\psi$ maps an irregular grid covering $\vs$ into a regular, square grid covering $\tilde\vs$. Introducing a new defining function $\tilde\Phi(x,\tilde y):=\Phi(x,\psi^{-1}(\tilde y))$, we can transform any smoothly sampled data set into the one with a square grid. Clearly, if $\Phi$ satisfies DF1--DF4, then $\tilde\Phi$ satisfies DF1--DF4 as well.

Conditions DF1--DF4 imply that (cf. \cite{belk} and \cite{trev2}, Sections VIII.5 and VIII.6):
\begin{enumerate}
\item The GRT $\R$ is a Fourier Integral Operator (FIO) with phase function $\la\Phi(x,y)$;
\item The corresponding canonical relation is 
\be\label{can-rel}
C:=\{((x,\la\Phi_x'(x,y)),(y,-\la\Phi_y'(x,y)):\, \Phi(x,y)=0,\la\in\br\setminus0,x\in\us,y\in\vs\},
\ee
which is a local canonical graph;
\item Any suitably modified adjoint of $\R$, denoted $\R^*$, is also an FIO, whose canonical relation $C^*$ is obtained from \eqref{can-rel} by switching the $(x,\xi)\in T^*\us$ and $(y,\eta)\in T^*\vs$ variables; and
\item The composition $\R^*(\dots)\R$, where the dots denote a cut-off combined with a suitable differential operator, is a pseudo-differential operator ($\Psi$DO), i.e. $C^*\circ C$ is a subset of the diagonal in $T^*\us$.
\end{enumerate}

Given a point $x\in\s$, find $y=y(x)$ (which is smooth locally) such that $\s_y$ is tangent to $\s$ at $x$. Denote
\be\label{sffdiff}
N(x):=\text{II}_{\s_y}(x)-\text{II}_{\s}(x),
\ee
where $\text{II}_{\s}(x)$ is the matrix of the second fundamental form of $\s$ at $x\in \s$ written in an orthonormal basis of $T_x\s$.

For any $x\in\us$, introduce the set 
\be\label{ga-def}
\Gamma_x:=\{y\in\vs:\, x\in \s_y,\,\s_y\text{ is tangent to $\s$ at some $z$},\ z\not=x\}. 
\ee

\begin{definition}\label{gpnotonS} A pair $(x_0,y_0)\in\ir$ is {\bf globally generic} if whenever $\s_{y_0}$ is tangent to $\s$ at some $z\not=x_0$ the following conditions hold:
\begin{itemize}
\item[GG1.] $\s$ is smooth at $z$, and $N(z)$ is either positive or negative definite; and
\item[GG2.]\label{gen3} Let $\dot\Gamma_{x_0}$ be a non-vanishing at any point tangent vector field along $\Gamma_{x_0}$. There exists an open set $\vs_1$, $y_0\in\vs_1\subset\vs$, such that for each $m\in\mathbb Z^3$, $|m|>0$, and all $\de>0$ sufficiently small, 
\begin{enumerate}
\item The set $\{y\in\Gamma_{x_0}\cap\vs_1:\, |m\cdot \dot\Gamma_{x_0}(y)|\le\de\}$ is contained in a finite number of segments of $\Gamma_{x_0}$ (this number may depend on $m$ and $\de$), and 
\item The sum of the lengths of these segments goes to zero as $\de\to0$.  
\end{enumerate}
\end{itemize}
\end{definition}
\noindent
As is shown in Section~\ref{remote}, Conditions GG1 and DF3 imply that $\Gamma_{x_0}$ is a smooth curve, so Condition GG2 makes sense.

An example when Condition GG2 is violated is when $\Gamma_{x_0}$ contains a straight line segment and $m\cdot\dot\Gamma_{x_0}(y)\equiv0$ on this segment for some $m\in\mathbb Z^3$, $|m|>0$.

\begin{definition}\label{gponS} A pair $(x_0,y_0)\in\ir$, $x_0\in \s$, is {\bf locally generic} if whenever $\s_{y_0}$ is tangent to $\s$ at $x_0$ the following conditions hold: 
\begin{itemize}
\item[LG1.] $S$ is smooth at $x_0$, and $N(z)$ is either positive or negative definite; and
\item[LG2.] There is no $\la\not=0$ such that $\la\Phi'_y(x_0,y_0)\in\mathbb Z^3$.
\end{itemize}
\end{definition}

\begin{definition} A pair $(x_0,y_0)\in\ir$ is {\bf generic} if it is both locally and globally generic. 
\end{definition}
\noindent

Let $\varphi$ be an interpolating kernel (IK), i.e. $\varphi(0)=1$ and $\varphi(j)=0$ for all $j\in\mathbb Z$, $j\not=0$. Suppose also that $\varphi$ satisfies the following assumptions:
\begin{itemize}
\item[IK1.] $\varphi$ is exact up to the order $2$, i.e.
\be\label{ker-int}
\sum_{j\in \mathbb Z^3} j^m\varphi(u-j)=u^m,\quad m\in(0\cup\mathbb N)^3,\,|m|\le 2,\ u\in\br^3;
\ee
\item[IK2.] $\varphi$ is compactly supported;
\item[IK3.] All partial derivatives of $\varphi$ up to the order $2$ are continuous;
\item[IK4.] \label{Hilb-con} All partial derivatives of $\varphi$ of order $3$ are piecewise continuous and bounded; and
\item[IK5.] $\varphi$ is normalized, i.e. $\int\varphi(y)dy=1$.
\end{itemize}
The interpolated version of $g$ can be written in the form
\be\label{f-int}
g_\e(y):=\sum_{j\in r+\mathbb Z^3} g(\e j) \varphi\left(\frac{y-\e j}{\e}\right).
\ee

First, we derive a microlocal inversion formula for the GRT $\R$, which reconstructs exactly the leading singularities of $f$. Pick any $(x_0,y_0)\in\ir$. Let $\al_0$ be a unit vector normal to $\s_{y_0}$ at $x_0$. For $(x,\al)\in\us\times S^2$ close to $(x_0,\al_0)$ and for $t$, $|t|\ll 1$, find the local solution $y=Y(\al,t;x)$ such that $x+t\al\in\s_y$ and $\al$ is normal to $\s_y$ at $x+t\al$. By construction, $y_0=Y(\al_0,t=0;x_0)$. Here we use the assumption that the data are complete, i.e. such a solution exists. It is shown below (see \eqref{Y-prp}) that the map $(\al,t)\to y=Y(\al,t;x)$ is a local diffeomorphism that depends smoothly on $x$. 

%Fix any reconstruction point $x\in \us$ and direction $\al\in S^2$. Pick any $t$, with $|t|$ sufficiently small, and find local solutions $y=Y_j(\al,t;x)$, $j=1,2,\dots,J$, such that $\s_y$ contains the point $x+t\al$ and the vector $\al$ is normal to $S_y$ at that point.  Microlocally, in a neighborhood of any $(x,\al)\in \us\times S^2$ there can be multiple solutions $Y_j$. Additionally, the number of solutions $J$ may depend on direction and vary from point to point, i.e. $J=J(x,\al)$. We can account for this situation by introducing a normalized smooth weight function $w_j(x,\al)$:
%\be\label{weight}
%\sum_j w_j(x,\al)\equiv 1,\ x\in \us,\al\in S^2,\ w_j\in C^{\infty}( \us\times S^2),
%\ee
%which can be constructed using a partition of unity. The inversion formula with continuous data is given by
%\be\label{inv-for-prelim}\begin{split}
%f_1(x)&=-\frac1{8\pi^2}\sum_j\int_{S^2}\frac{w_j(x,\al)} {b(x,Y_j(\al,0;x))}\left.\left(\frac{\pa}{\pa t}\right)^2 g(Y_j(\al,t;x))\right|_{t=0} d\al.
%\end{split}
%\ee
%To avoid unnecessary complications that would obscure the main ideas of the paper, in what follows we will consider only one of the solutions, drop the subscript $j$ from notation, and drop the weight from \eqref{inv-for-prelim}.
%Clearly, a partition of unity can be selected so that for all $j$:
%\begin{enumerate}
%\item $w_j$ is even: $w_j(\al,x)=w_j(-\al,x)$,
%\item the solution $Y_j$ is even: $Y_j(\al,t;x)=Y_j(-\al,-t;x)$, and 
%\item $Y_j$ is smooth on the support of $w_j$. 
%\end{enumerate}
%Given this convention, t
Let $\vs_1$ be a small neighborhood of $y_0$. Pick any $\chi\in\coi(\vs_1)$ such that $\chi\equiv1$ near $y_0$. The inversion formula with continuous data is given by
\be\label{inv-formula}\begin{split}
f_\chi(x)&=-\frac1{4\pi^2}\int_{S_+^2}\frac{\chi(Y(\al,0;x))}{b(x,Y(\al,0;x))}\left.\left(\frac{\pa}{\pa t}\right)^2 g(Y(\al,t;x))\right|_{t=0} d\al.
\end{split}
\ee
This inversion formula emulates the CRT inversion formula by backprojecting a second order derivative of the GRT. The affine variable $t$ is computed relative to $x$ as opposed to the origin, as is the case with the CRT. Hence the GRT analogue of the usual term $\al\cdot x$ is missing from \eqref{inv-formula}, because it is absorbed by the function $Y$. Due to the symmetry $g(Y(\al,t;x))=g(Y(-\al,-t;x))$, in \eqref{inv-formula} we integrate over half of the unit sphere $S_+^2$.
%In what follows, we will consider only one of the solutions $Y_j$ and drop the subscript $j$ from $Y_j$ and $w_j$ for simplicity. All other solutions can be considered completely analogously.

Using the argument following \eqref{data}, it is easy to show that the map $f\to f_\chi$ is a $\Psi$DO of degree zero with principal symbol 1 microlocally near $(x_0,\al_0)$ (see e.g. \cite{belk, kat10b}). 
%Our result in this paper is purely local, and the conditions we impose are local as well. They are described in terms of local behavior of $\Phi$. 
Thus, the singularities of $f$ and $f_\chi$ are the same to leading order (e.g., in the scale of Sobolev spaces) microlocally near $(x_0,\al_0)$. An inversion formula that recovers all the singularities of $f$ can be obtained by combining \eqref{inv-formula} with a microlocal partition of unity. In the case of discrete data, we use the same inversion formula \eqref{inv-formula}, but replace $g$ with $g_\e$. The corresponding reconstruction is denoted $f_{\chi\e}$.

\section{Statement of main result. Beginning of proof}\label{local-grt}

\subsection{Statement of main result.}
Pick $(x_0,y_0)\in\ir$ such that $\s$ is smooth at $x_0\in \s$, $\s_{y_0}$ is tangent to $\s$ at $x_0$, and $N(x_0)$ is either positive definite or negative definite. Fix some orthonormal basis in the common plane tangent to $\s$ and $\s_y$ at $x_0$. Let $\al_0$ be the unit vector normal to $\s$ at $x_0$. 
%Define
%\be\label{th0}
%\al_0:=\Phi'_x/|\Phi'_x|,\ \bt_0:=\Phi'_y/|\Phi'_y|.
%\ee
%Here and in what follows we use the convention that if the arguments of $\Phi$ and its derivatives are omitted, then they are evaluated at $(x_0,y_0)$. 
%and $\al_0\cdot (x-x_0)\ge 0$ for all $x\in S$ close to $x_0$. Find $y_0$ such that 
%\be\label{main-eq}
%\Phi(x_0,y_0)=0,\ \Phi'_x(x_0,y_0) \parallel \al_0.
%\ee
%Without loss of generality we may assume that $x_0=0$ and $y_0=0$. Also, without loss of generality, we may suppose that the support of $w$ is contained in a small neighborhood of $(x_0,\pm\al_0)$. 
%Introduce $x$- and $y$-coordinates with the origins at $x_0$ and $y_0$, respectively, so that
%\be\label{rotate}
%\begin{split}
%\al&=(\al_1,\al^\perp)\in S^2,\ \al_0=(-1,0,\dots,0);\\
%\bt&=(\bt_1,\bt^\perp)\in S^2,\ \bt_0=(1,0,\dots,0).
%\end{split}
%\ee
%The equation of the plane $T_{x_0}S$ is thus $\al_1=0$. Multiplying, if necessary, $\Phi$ by $-1$, 
For convenience of calculations, we sometime use the Cartesian coordinates $(x_1,x^\perp)$ determined by
\be\label{rotate-xy}
\begin{split}
x=x_1\al_0+x^\perp,\ x^\perp\in\al_0^\perp,
\end{split}
\ee
where $\al_0^\perp$ is the plane through $x_0$ and normal to $\al_0$. This plane is tangent to both $\s$ and $\s_{y_0}$ at $x_0$. The direction of $\al_0$ is chosen so that $N(x_0)$ is {\it negative definite}. The side of $\s$ where $\al_0$ points is called {\it interior}. The other side of $\s$ is called {\it exterior}. If necessary, multiply $\Phi$ by $(-1)$ so that $\Phi'_x(x_0,y_0)/|\Phi'_x(x_0,y_0)|=-\al_0$. 

Consider the point
\be\label{rec-pt}
\xh:=x_0+\e \tilde x,\ \tilde x\in\tilde\us,
\ee
where $\tilde \us$ is a bounded set. Denote:
\be\label{f0-def}\begin{split}
f_\chi(x_{0^{\pm}})&:=\lim_{\e\to 0^{\pm}}f_\chi(x_0+\e \al_0),\ f_0:=\lim_{\e\to 0^+}(f(x_0+\e \al_0)-f(x_0-\e \al_0)),\\
\mc&:=\frac{|\Phi'_x|}{|\Phi'_y|},\ \bt_0=\frac{\Phi'_y}{|\Phi'_y|}.
\end{split}
\ee
Here and in what follows the convention is that if the arguments of $\Phi$ and its derivatives are omitted, then they are evaluated at $x_0,y_0$. 

We also introduce local $y$-coordinates with the origin at $y_0$: 
\be\label{bt0def}
y=(y_1,y^\perp)=y_1\bt_0+y^\perp.
\ee
Thus, equation $y_1=0$ determines the plane tangent to the submanifold $\T_{x_0}$ at $y_0$. We frequently denote this plane $\bt_0^\perp$.

Finally, we use the extension of the CRT to all of $\br^3\setminus 0$ according to:
\be\label{crt-def}
\hat f(u,s):=\int f(x)\de(u\cdot x-s)dx,\ u\in\br^3\setminus 0,
\ee
for a sufficiently smooth $f$.

The main result of the paper is the following
\begin{theorem}\label{thm:main} Pick a generic pair $(x_0,y_0)\in\ir$ such that $\s_{y_0}$ is tangent to $\s$ at $x_0\in \s$. Then
\be\label{main-res}
\lim_{\e\to0}f_{\chi\e}(\xh) = f_\chi(x_{0^+})-f_0\int_{\mc h}^\infty \hat\varphi(\bt_0,s) d s,
\ee
where $h=\tilde x\cdot\al_0$, and $\hat\varphi$ is the CRT of $\varphi$.
\end{theorem}

By IK5, $\int_{\br}\hat\varphi(\bt_0,s) d s=1$. The inversion formula \eqref{inv-formula} reconstructs jumps of $f$ accurately, so $f_\chi(x_{0^{+}})-f_\chi(x_{0^{-}})=f_0$ and \eqref{main-res} can be written as follows
\be\label{main-res-v2}
\lim_{\e\to0}f_{\chi\e}(\xh) = f_\chi(x_{0^-})+f_0\int^{\mc h}_{-\infty} \hat\varphi(\bt_0,s) d s.
\ee

By linearity, the proof of Theorem~\ref{thm:main} can be split into two parts: local and global. The local part is formulated as follows.
\begin{theorem}\label{thm:local} Pick a locally generic pair $(x_0,y_0)\in\ir$ such that $\s_{y_0}$ is tangent to $\s$ at $x_0\in \s$. Suppose $\text{supp}(f)$ is contained in a sufficiently small neighborhood of $x_0$. Then
\be\label{main-res}
\lim_{\e\to0}f_{\chi\e}(\xh) = f_\chi(x_{0^+})-f_0\int_{\mc h}^\infty \hat\varphi(\bt_0,s) d s,
\ee
where $h=\tilde x\cdot\al_0$, and $\hat\varphi$ is the CRT of $\varphi$.
\end{theorem}

Since $\R$ is an FIO with canonical relation \eqref{can-rel}, $g=\R f$ is singular only when $\s_y$ is tangent to $\s$. Therefore, we are interested in the behavior of $g$ in a neighborhood of $y_0$. 
%The behavior of $g$ is not affected by how $S$ is parametrized, so we will choose a convenient parametrization. 
%Suppose first that $x_0$ is the only point where $\s_{y_0}$ is tangent to $\s$.
\begin{remark} Strictly speaking, one has to distinguish between the original coordinates that describe points $x\in\us$, $y\in\vs$ and those in \eqref{rotate-xy}, \eqref{bt0def}, respectively. For example, one should write $x=\hat x_1\al_0+\hat x^\perp$ instead of $x=x_1\al_0+x^\perp$. Such notation would emphasize that $\hat x_1$ is not the first component of $x$ in the original coordinates, i.e. $\hat x_1\not=x_1$. Similarly, a derivative like $\pa\Phi(x,y)/\pa x_1$, if written in full, becomes $\pa\Phi(x(\hat x_1,\hat x^\perp),y(\hat y_1,\hat y^\perp))/\pa \hat x_1$. To avoid burdensome notations, whenever the coordinates $(\hat x_1,\hat x^\perp)$ and $(\hat y_1,\hat y^\perp)$ are used, we will stick with the simplified notation and assume that the above convention holds.
\end{remark}

\subsection{Behavior of $g$ near its singular support}\label{ssgb}

Because $\s$ is smooth in a neighborhood of $x_0\in \s$, there is a smooth local diffeomorphism $x\to (z,p)$ so that 
\be\label{locdiff} 
x=z+p\nal(z),\ z\in \s,\ \nal(z) \text{ is normal to $\s$ at $z$},\ |\nal(z)|\equiv1.
\ee
The normal $\nal(z)$ is chosen so that $N(z)$ is negative definite. Thus $\nal(x_0)=\al_0$. Clearly, we can extend the function $\nal(z)$, $z\in\s$, to $\nal(x)$ defined in a neighborhood of $\s$ by the formula $\nal(z+p\nal(z)):=\nal(z)$. With a slight abuse of notation, the extended function will also be denoted $\nal(\cdot)$.

Using \eqref{locdiff}, define $\Psi(z+p\nal(z)):=p$. Then $\Psi(x)=0$ is the equation of $\s$ near $x_0$, and $\Psi$ is smooth. By construction, $\Psi(x)>0$ on the interior side of $\s$.

Consider the system of equations
\be\label{system-2}
\Phi'_x(z+p\nal(z),y)-\la \nal(z)=0,\ \Phi(z+p\nal(z),y)=0,\ \Psi(z)=0.
\ee
If we set $p=0$ and solve \eqref{system-2} for $y$, we find submanifolds $\s_y$ tangent to $\s$. We also need to solve these equations for $z,p$, and $\la$ in terms of $y$.

%In what follows we use the convention that if the arguments of $\Phi$ and its derivatives are omitted, then they are evaluated at $x_0,y_0$. 

\begin{lemma}\label{ssg-smooth} Pick $(x_0,y_0)\in\ir$ such that (1) $\s_{y_0}$ is tangent to $\s$ at $x_0\in \s$, (2) $\s$ is smooth at $x_0$, and (3) $N(x_0)$ is negative definite. There exists an open set $\vs_1$, $y_0\in\vs_1\subset\vs$, such that
\begin{enumerate} 
\item The set of $y\in\vs_1$ such that $\s_y$ is tangent to $\s$ in a neighborhood of $x_0$ is a smooth submanifold of $\vs$ through $y_0$. The vector $\Phi'_y(x_0,y_0)$ is normal to this submanifold at $y_0$.
\item The solutions $z=Z(y)$, $p=P(y)$, and $\la=\Lambda(y)$ to \eqref{system-2} depend smoothly on $y\in\vs_1$, and 
\be\label{pprime}
P'_y(y_0)=\frac1{|\Phi_x'|}\Phi_y'\not=0;
\ee
\item Equations \eqref{system-2} determine a smooth function $y=\oy(z,p)$, $(z,p)\in\s\times\br$, in a neighborhood of $(x_0,0)$.
\end{enumerate}
\end{lemma}
\begin{proof} 
Differentiate \eqref{system-2} with respect to $z$, $p$, $\la$, and $y$, and set $z=x_0$, $p=0$, $y=y_0$ to obtain the $5\times 8$ matrix
\be\label{keym_tan-v2}
\begin{bmatrix} 
\Phi''_{xx}-\la\nal'_x & \Phi''_{xx}\al_0 & -\al_0 & \Phi''_{xy}\\
\Phi'_x  & \Phi'_x\cdot\al_0 & 0 & \Phi'_y\\
\Psi'_x & 0 & 0 & 0
\end{bmatrix}.
\ee
Here $n_x'$ is the derivative of the function $n(x)$ extended to a neighborhood of $\s$ as described following \eqref{locdiff}. Since $\s_{y_0}$ is tangent to $\s$ at $x_0\in \s$, we have $\Phi'_x\parallel \al_0$, so $\Phi'_x\cdot\al_0=\la=-|\Phi'_x|\not=0$. This also gives the value of $\la$ to be used in \eqref{keym_tan-v2}.

To prove the first part of the first assertion we need to show that $z,\la$, and $y_1$ are smooth functions of $y^\perp$. Remove the columns corresponding to the derivatives with respect to $p$ (because $p=0$ is fixed) and $y^\perp$ to obtain a $5\times 5$ submatrix. Calculation in coordinates shows that $\Psi_x'(z+p\nal(z))\equiv\nal(z)$ (cf. \eqref{locdiff}). By applying elementary row and column operations, it is clear that this submatrix is full-rank if and only if the following matrix has rank two:
\be\label{submatr}
\Phi''_{x^\perp x^\perp}-\la(\nal^\perp)'_{x^\perp}.
\ee
Here $\Phi''_{x^\perp x^\perp}:\al_0^\perp\to \al_0^\perp$ is the appropriate submatrix of $\Phi''_{xx}$ in the coordinates \eqref{rotate-xy}, and $n^\perp$ is the projection of $n$ onto $\al_0^\perp$. It is easy to see that 
\be\label{submatr-v2}
\Phi''_{x^\perp x^\perp}-\la(\nal^\perp)'_{x^\perp}=-\la N(x_0).
\ee
The desired assertion follows from Condition LG1 (see Definition~\ref{gponS}). 

Next, set $p=0$ in \eqref{system-2} and assume that $z,y_1$ are functions of $y^\perp$. Differentiating the last two equations in \eqref{system-2} with respect to $y^\perp$ and using that $\Phi'_y\cdot\bt_0=|\Phi_y'|\not=0$ gives $\Phi'_x z'_{y^\perp}=0$ and $\pa y_1/\pa y^\perp=0$. This proves the second part of the first assertion.

The first part of the second assertion follows by retaining the columns corresponding to the derivatives with respect to $z,p$, and $\la$. As before, the resulting $5\times 5$ submatrix is full-rank because the  matrix in \eqref{submatr} has rank two. The second part of the second assertion follows by considering $z,p$, and $\la$ as functions of $y$, differentiating the last two equations in \eqref{system-2} with respect to $y$, and using that $\Phi'_x,\Phi'_y\not=0$.

The third assertion follows immediately from data completeness and the Bolker condition (Conditions DF2 and DF4, respectively).
\end{proof}

%We have
%\be\label{lead-t-alt-2}\begin{split}
%g(y)= \int f(x)b(x,y)\theta(S(x))\de(\Phi(x,y)) |\Phi'_x(x,y)|dx.
%\end{split}
%\ee

%In what follows, for any $w\in S$ and $y$ such that $S_y$ is tangent to $S$ at $w$, the matrix-function $N(w)$ denotes the difference of matrices of second fundamental forms of $S_y$ and $S$ at $w$.

We need the following lemma, which generalizes one of the results of Ramm and Zaslavsky \cite{rz2, rz1} from the CRT to the GRT.

\begin{lemma}\label{der-g}
Pick $(x_0,y_0)\in\ir$ such that (1) $\s_{y_0}$ is tangent to $\s$ at $x_0\in \s$, (2) $\s$ is smooth at $x_0$, and (3) $N(x_0)$ is negative definite. Suppose $\text{supp}(f)$ is contained in a small neighborhood of $x_0$. For any $z\in \s$ and $p$ in small neighborhoods of $x_0$ and $0$, respectively, one has:
\be\label{GRT-loc-v3}
g(\oy(z,p))=p_+G(z,p)+G_1(z,p) \text{ and } G(z,0)= f_0(z) b(z,\oy(z,0))\frac{2\pi}{\sqrt{\text{det}N(z)}},
\ee
for some smooth $G(z,p)$, $G_1(z,p)$. 
%If $f\equiv0$ on the exterior side of $\s$ near $x_0$, then $G_1(z,p)\equiv0$ near $(x_0,p)$.
\end{lemma}

\begin{proof}
Recall that $\oy(z,p)$ is the smooth function of $(z,p)\in\s\times\br$ defined by the conditions that $z+p\nal(z)\in \s_y$ and $\nal(z)$ be normal to $\s_y$ at the point $z+p\nal(z)$, see Assertion (3) of Lemma~\ref{ssg-smooth}. Recall also that in coordinates \eqref{rotate-xy}, $N(x_0)$ is negative definite.  By linearity, we may assume that 
%$\text{supp}(f)$ is contained in a small neighborhood of $x_0$ and 
$f\equiv0$ on the exterior side of $\s$. In particular, $f(x_0-\e \al_0)\equiv0$,\ $\e>0$, in \eqref{f0-def}. In this case we have to prove \eqref{GRT-loc-v3} with $G_1\equiv0$. By construction,
\be\label{lead-t-alt-1}\begin{split}
&g(y)= \int_{\s_y} f(x)b(x,y)\theta\left(\Psi(x)\right) dx,
\end{split}
\ee
where $\theta$ is the unit step function (Heaviside function). Consider the system 
\be\label{system-3}
\Psi'_x(x) -\mu \Phi'_x(x,y)=0,\ \Phi(x,y)=0,
\ee
which we solve for $x$ and $\mu$ in terms of $y$. Equations \eqref{system-3} determine the stationary point $x_*(y)$ of $\Psi(x)$ on the surface $\s_y$ (the parameter $\mu$, which corresponds to $1/\la$ in \eqref{system-2}, is the Lagrange multiplier). From  \eqref{system-2} and the property of $\Psi'_x$, the solution $x_*(y)$, $\mu(y)$ to \eqref{system-3} can be obtained from the solution $z=Z(y),p=P(y),\la=\Lambda(y)$ to \eqref{system-2}: $x_*(y)=Z(y)+P(y)\nal(Z(y))$, $\mu(y)=1/\Lambda(y)$. By Assertion (2) of Lemma~\ref{ssg-smooth}, $x_*(y)$ is smooth near $y_0$. 

As is easily checked, when the matrix $\Psi''_{xx}(x_0) -\mu(y_0) \Phi''_{xx}(x_0,y_0)$ is viewed as a quadratic form on $\al_0^\perp$, it coincides with $N(x_0)$. The latter is negative definite, hence $x_*(y)$ is the local maximum of $\Psi(x)$ on $\s_y$. By the Morse lemma, find local coordinates $\omega$ on $\s_y$, which depend smoothly on $y$, such that $\omega(x_*(y))=0$ and $\Psi(x)=\Psi(x_*(y))-|\omega|^2$, $x=x(\omega;y)\in \s_y$. Since $f,b$, and $x(\omega;y)$ are all smooth, we get from \eqref{lead-t-alt-1}
\be\label{gy-st1}
g(y) = \int \theta\left(\Psi(x_*(y))-|\omega|^2\right) F(\omega,y)d\omega
\ee
for some smooth $F$. Expand $F$ in the Taylor series around $\omega=0$ and integrate in spherical coordinates $\omega=r\Theta$. Integration with respect to $\Theta$ removes all the odd powers of $r$, i.e. only the even powers of $r$ remain. By construction, 
\be\label{Sx-val}
\Psi(x_*(y))=\Psi(z+p\nal(z))=p,\ y=\oy(z,p),
\ee
and the first statement in \eqref{GRT-loc-v3} follows.

To prove the second statement, we use the local coordinates \eqref{rotate-xy}. In these coordinates, the local equation of $\s_y$ becomes
\be\label{sy-loc-1}
x_1=x_1(x^\perp,y)=\phi(y)+a(y)\cdot x^\perp+\frac{A(y)x^\perp\cdot x^\perp}2+O(|x^\perp|^3)
\ee
for some smooth $\phi$, $a$, and $A$. Suppose $y$ is such that $x_*(y)=x_0+p\al_0$. Substitute such a pair $(x_*(y),y)$ into \eqref{system-3} and use \eqref{sy-loc-1} to conclude that $\phi(y)\equiv p$ and $a(y)\equiv0$. Let $x_1=Q(x^\perp)$ be the equation of $\s$ in the coordinates \eqref{rotate-xy}. By construction, $Q'(0)=0$. Clearly, $A(y_0)$ and $Q''(0)$ are the matrices of the second fundamental form of $\s_{y_0}$ and $\s$, respectively, at $x_0$ in the coordinates \eqref{rotate-xy}. From \eqref{lead-t-alt-1}, 
\be\label{lead-t-alt-2}\begin{split}
&g(y)= \int f(x)b(x,y)\theta\left(\psi(x^\perp,y)\right)
\sqrt{1+|Q'(x^\perp)|^2}\, dx^\perp,\\ 
&\psi(x^\perp,y):=x_1(x^\perp,y)-Q(x^\perp),\ x=(x_1(x^\perp,y),x^\perp).
\end{split}
\ee
Integrating in \eqref{lead-t-alt-2} by diagonalizing $N=\psi''_{x^\perp x^\perp}(0,y_0)$ and changing variables, we find 
\be\label{lead-t-v3}\begin{split}
G(x_0,0)= f(x_0) b(x_0,y_0)\frac{2\pi}{\sqrt{\text{det}N}},
\end{split}
\ee
which finishes the proof.
\end{proof}

\section{Local behavior of interpolated data.}\label{locint}

Similarly to \eqref{system-2}, consider the equations 
\be\label{system-Y}
\Phi'_x(x+t\al,y)-\la \al=0,\ \Phi(x+t\al,y)=0,
\ee
which we solve to find $y=Y(\al,t;x)$ assuming $(\al,t,x)$ is in a neighborhood of $(\al_0,0,x_0)$. Differentiating \eqref{system-Y} with respect to $\la$ and $y$ and setting $x=x_0$, $t=0$, $\al=\al_0$, $\la=-|\Phi'_x|$, we obtain similarly to \eqref{keym_tan-v2} a matrix, which is non-degenerate. 
As opposed to \eqref{keym_tan-v2}, the key reason why it is non-degenerate is the Bolker condition. Hence $Y(\al,t;x)$ is a smooth function of $\al,t$ and $x$. Next we substitute $y=Y(\al,t;x)$ into \eqref{system-Y} and obtain a few useful properties of $Y$. The first one is that $\pa Y_1/\pa\al^\perp=0$. Indeed, differentiate the second equation in \eqref{system-Y} with respect to $\al^\perp$ and set $x=x_0$, $t=0$, $\al=\al_0$ to obtain $\Phi'_y\pa Y/\pa\al^\perp=0$. The desired assertion follows from \eqref{bt0def}. In a similar fashion, we have
\be\label{Y-prp}
\det(\pa Y^\perp/\pa\al^\perp)\not=0,\ \pa Y_1/\pa t=|\Phi_x'|/|\Phi_y'|\not=0,\ 
\det(\pa Y/\pa(\al^\perp,t))\not=0.
\ee
The first result is obtained by differentiating the first equation in \eqref{system-Y} with respect to $\al^\perp$ and using the Bolker condition and that $\pa Y_1/\pa\al^\perp=0$. The second result is obtained by differentiating the second equation in \eqref{system-Y} with respect to $t$ and using the properties of the selected coordinates \eqref{rotate-xy}, \eqref{bt0def}. The last result is an obvious consequence of the first two and that $\pa Y_1/\pa\al^\perp=0$.

In view of \eqref{Y-prp}, given any small $\omega>0$, we can find a sufficiently small open set $\vs_1$, $y_0\in\vs_1\subset\vs$, such that $Y(\al,t;\xh)\in \vs_1$ implies $|\al^\perp|<\omega$ for all $t=O(\e)$ and $\xh$ provided that $\e$ is sufficiently small. This relationship between $\vs_1$ and $\omega$ is assumed in what follows. 

Since $\text{supp}(\chi)\subset\vs_1$, the integral with respect to $\al$ in \eqref{inv-formula} can be split into two sets:
\be\label{sets} \begin{split}
\Omega_1&:=\{\al\in S_+^2:\,|\al^\perp|<A\e^{1/2}\},\ \Omega_2:=\{\al\in S_+^2:\,A\e^{1/2}<|\al^\perp|<\om\},
\end{split}
\ee
for some small (but fixed) $\om>0$. Here $A>0$ is a large parameter. Let $f_{\chi\e}^{(j)}$ denote the result of integrating in \eqref{inv-formula} (with $g$ replaced by $g_\e$) over $\Omega_j$, $j=1,2$. The behavior of $f_{\chi\e}^{(1)}$ is investigated first. This is done in Section~\ref{piece-one}. In the remainder of this section we lay the groundwork for that investigation by deriving the behavior of $Y(\al,t;\xh)$ and $g_\e(y)$ in a neighborhood of $(\al_0,0,x_0)$ and $y_0$, respectively.

\subsection{Local behavior of $Y(\al,t;\xh)$.} The first step is to obtain the leading term behavior of the function $Y(\al,t;\xh)$ for $t=O(\e)$ and $\al\in\Omega_1$, i.e. for $|\al^\perp|=O(\e^{1/2})$. In this section we continue using the coordinates \eqref{rotate-xy}. 

Expanding $y=Y(\al,t;\xh)$ in the Taylor series around $x=x_0$, $t=0$, $\al=\al_0$ and using that $|\xh-x_0|=O(\e)$, $t=O(\e)$, $|\al^\perp|=O(\e^{1/2})$, and $\pa Y_1/\pa\al^\perp=0$ gives 
\be\label{y-coord}
y_1=O(\e),\ |y^\perp|=O(\e^{1/2}).
\ee
For $x$ in an $O(\e)$ neighborhood of the origin (i.e., $x_0$) and for $y$ in an $O(\e^{1/2})$ neighborhood of the origin (i.e., $y_0$) we have
\be\label{expansion}\begin{split}
\Phi(x,y)&=\Phi'_x\cdot x+\Phi'_y\cdot y+\frac{\Phi''_{yy}y\cdot y}2+O(\e^{3/2}),\\
\Phi'_x(x,y)&=\Phi'_x +\Phi''_{xy}y+O(\e).
\end{split}
\ee
%Here and in what follows we use the convention that if the arguments of $\Phi$ and its derivatives are omitted, then they are evaluated at $(x_0,y_0)$. 
To find $y=Y(\al,t;\xh)$, substitute $x=\xh$ into \eqref{expansion} and solve
\be\label{expansion-st2}\begin{split}
&\Phi'_x \cdot (\e\tilde x+t\al)+\Phi'_y\cdot y+\frac{\Phi''_{yy}y\cdot y}2=O(\e^{3/2}),\\
&\Phi'_x + \Phi''_{xy} y =\la \left(1,\al^\perp\right) \ (\text{mod }O(\e)).
\end{split}
\ee
Recall that $h=\tilde x\cdot\al_0$. Switching to the coordinates \eqref{rotate-xy}, \eqref{bt0def}, using \eqref{y-coord}, and keeping only the terms of order $O(\e)$ in the first equation in \eqref{expansion-st2} gives
\be\label{exp-eq1}\begin{split}
-(\e h+t)|\Phi'_x|+|\Phi'_y| y_1+\frac{\Phi''_{y^\perp y^\perp}y^\perp\cdot y^\perp}2=O(\e^{3/2}).
\end{split}
\ee

Projecting the second equation in \eqref{expansion-st2} onto $\al_0$ and onto $\al_0^\perp$ implies 
\be\label{exp-eq2al0}
\la=-|\Phi'_x| +O(\e^{1/2}),\ 
\Phi''_{x^\perp y^\perp}y^\perp=\la \al^\perp+O(\e),
\ee
leading to
\be\label{yprp-sol}
Y^\perp(\al,t;\xh)=-|\Phi'_x|(\Phi''_{x^\perp y^\perp})^{-1}\al^\perp+O(\e).
\ee
By the Bolker condition \eqref{bolker}, $\Phi''_{x^\perp y^\perp}$ is nondegenerate. Substitution into \eqref{exp-eq1} now yields:
\be\label{y1-sol}\begin{split}
Y_1(\al,t;\xh)&=\mc\left(\e h+t-\frac{M\al^\perp\cdot\al^\perp}2\right)+O(\e^{3/2}),\\
M:&=|\Phi'_x|(\Phi''_{x^\perp y^\perp })^{-T}\Phi''_{y^\perp y^\perp }(\Phi''_{x^\perp y^\perp })^{-1},\ M:\al_0^\perp\to\al_0^\perp.
\end{split}
\ee
Recall that $\mc$ is defined in \eqref{f0-def}. 

\subsection{The leading local behavior of $g(Y(\al,t;\xh))$}\label{GRT-near}
We plan to substitute $y=Y(\al,t=0;\xh)$ into \eqref{f-int}. Hence the second step is to find the leading behavior of $g(y)$ when $|y-Y(\al,0;\xh)|=O(\e)$ and $|\al^\perp|=O(\e^{1/2})$. This is done by finding the asymptotics of $z=Z(y)$ and $p=P(y)$, which are determined by solving \eqref{system-2}. 

Recall that the local equation of $\s$ in the coordinates \eqref{rotate-xy} and the interior unit normal are given by
\be\label{S-eq}
z_1=\frac{Qz^\perp\cdot z^\perp}{2}+O(|z^\perp|^3),\ 
\nal(z)=(1+O(|z^\perp|^2),-Qz^\perp+O(|z^\perp|^2)).
\ee
By \eqref{yprp-sol}, \eqref{y1-sol}, $|y_1|=O(\e)$, $|y^\perp|=O(\e^{1/2})$. From \eqref{pprime}, $\pa p/\pa y^\perp=0$ at $y=y_0$, so this implies $|p|=O(\e)$, $x_1=O(\e)$, and $|z^\perp|,|x^\perp|=O(\e^{1/2})$, where $x=z+pn(z)$. The equation \eqref{S-eq} leads to:
\be\label{y-cond}\begin{split}
x&=\left(\frac{Qz^\perp\cdot z^\perp}{2}+O(\e^{3/2}),z^\perp\right)+p(1+O(\e^{1/2}),-Qz^\perp+O(\e))\\
&=\left(p+\frac{Qz^\perp\cdot z^\perp}{2}+O(\e^{3/2}),z^\perp+O(\e^{3/2})\right).
\end{split}
\ee
From \eqref{y-cond}, $x^\perp=z^\perp+O(\e^{3/2})$.

Given that now $|x^\perp|=O(\e^{1/2})$, the expansion in the first line in \eqref{expansion} should include additional terms. The second equation in \eqref{system-2} becomes
\be\label{inter}
-|\Phi'_x|x_1+|\Phi'_y|y_1+\frac{\Phi''_{x^\perp x^\perp }x^\perp\cdot x^\perp}2+\Phi''_{x^\perp y^\perp }y^\perp\cdot x^\perp+\frac{\Phi''_{y^\perp y^\perp }y^\perp\cdot y^\perp}2=O(\e^{3/2}).
\ee
Solving for $x_1$ we find:
\be\label{x1-xp}
x_1=\frac1{\mc}y_1+\frac1{|\Phi'_x|}\left\{\frac{\Phi''_{x^\perp x^\perp }x^\perp\cdot x^\perp}2+\Phi''_{x^\perp y^\perp }y^\perp\cdot x^\perp+\frac{\Phi''_{y^\perp y^\perp }y^\perp\cdot y^\perp}2\right\}+O(\e^{3/2}).
\ee
Using \eqref{yprp-sol} and that $|y-Y(\al,0;\xh)|=O(\e)$, we have 
\be\label{ysm-prp}
y^\perp=-|\Phi'_x|(\Phi''_{x^\perp y^\perp })^{-1}\al^\perp+O(\e),
\ee
and the unit normal vector to $\s_Y$ is thus:
\be\label{x1-xp-norm}
\left(1+O(\e^{1/2}),\al^\perp-\frac{\Phi''_{x^\perp x^\perp }x^\perp}{|\Phi'_x|}+O(\e)\right).
\ee
The big-$O$ terms in \eqref{x1-xp-norm} follow by noticing that differentiation with respect to $x_1$ in \eqref{inter}, \eqref{x1-xp} converts $O(\e^{3/2})$ into $O(\e^{1/2})$, and differentiation with respect to $x^\perp$ converts $O(\e^{3/2})$ into $O(\e)$. This follows by looking at the terms absorbed by $O(\e^{3/2})$.

From \eqref{S-eq} and \eqref{x1-xp-norm}, the two normal vectors are parallel (the first equation in \eqref{system-2}) if
\be\label{par-conseq}
\al^\perp-\frac{\Phi''_{x^\perp x^\perp }z^\perp}{|\Phi'_x|}+O(\e)=-Qz^\perp+O(\e),
\ee
which implies
\be\label{zprpal}
z^\perp=N^{-1}\al^\perp+O(\e),\ x^\perp=N^{-1}\al^\perp+O(\e).
\ee
%According to the derivation in Section~\ref{ssgb}, we have to find $z^\perp$ and $p$ from the conditions that the point $x=z+p\nal(z)$, $z\in \s$, be on $\s_Y$ and the vector $\nal(z)$ be normal to $\s_Y$ at that point. Thus, $(z,p)$ here are essentially the same as in \eqref{system-2}. As was done before, assuming that $z,p$, and $\la$ are functions of $y$ and differentiating the second equation in \eqref{system-2} gives $\pa p/\pa y^\perp=0$ when $y=y_0$. By \eqref{y-coord}, this gives $p=O(\e)$ and $|z^\perp|=O(\e^{1/2})$. 
Matching the first components in \eqref{y-cond} and using \eqref{x1-xp}, \eqref{ysm-prp}, \eqref{zprpal} gives after simple transformations:
\be\label{p-st2}
p=\frac1{\mc}\left(y_1+\frac{M_1\al^\perp\cdot\al^\perp}2\right)+O(\e^{3/2}),\
M_1:=\mc(M-N^{-1}),
\ee
where $M$ is defined in \eqref{y1-sol}. Summarizing \eqref{zprpal} and \eqref{p-st2} we have:
\be\label{pzsum}\begin{split}
&Z^\perp(y)=N^{-1}\al^\perp+O(\e),\ P(y)=\frac1{\mc}\left(y_1+\frac{M_1\al^\perp\cdot\al^\perp}2\right)+O(\e^{3/2}),\\
&|y-Y(\al,0;\xh)|=O(\e),\ |\al^\perp|=O(\e^{1/2}).
\end{split}
\ee
Substituting into \eqref{GRT-loc-v3} we find 
\be\label{lead-t-v2}\begin{split}
g(y)= &f(Z(y)) b(Z(y),y^*)\frac{2\pi}{\sqrt{\text{det}N}} P_+(y)+O\left(P_+^2(y)\right),
\end{split}
\ee
where $y^*=\oy(z,p=0)$ whenever $y=\oy(z,p)$. 

From \eqref{S-eq}, \eqref{zprpal} and \eqref{yprp-sol}, \eqref{y1-sol} it follows that $|Z(y)-x_0|=O(\e^{1/2})$ and $|y-y_0|=O(\e^{1/2})$ whenever $|y-Y(\al,0;\xh)|=O(\e)$, hence
\be\label{fbappr}
f(Z(y)) b(Z(y),y^*)=f(x_0)b(x_0,y_0)+O(\e^{1/2}).
\ee

\subsection{Local behavior of the interpolated data.} 
In this subsection we find the behavior of the interpolated data near $y_0=0$. 
Combining \eqref{pzsum}--\eqref{fbappr} we get
\be\label{gej}\begin{split}
g(\e j)=&\frac{(f(x_0)b(x_0,y_0)+O(\e^{1/2}))2\pi/\mc}{\sqrt{\text{det}N}}
\left(\e j_1+\frac{M_1\al^\perp\cdot \al^\perp}2+O(\e^{3/2})\right)_+\\
&+O\left(P_+^2(\e j)\right),
\end{split}
\ee
where $M_1$ is defined in \eqref{p-st2}. From \eqref{y1-sol} and \eqref{pzsum}, $|P(\e j)|=O(\e)$.
Using \eqref{yprp-sol}, \eqref{y1-sol} and \eqref{gej} gives:
\be\label{f-int-Y}\begin{split}
g_\e&(Y(\al,t;\xh))\\
=&\frac{(f(x_0)b(x_0,y_0)+O(\e^{1/2})) 2\pi\e/\mc}{\sqrt{\text{det}N}}
\sum_{j\in r+\mathbb Z^3}\left(j_1+\frac{M_1 \tilde \al^\perp\cdot \tilde \al^\perp}2+O\left(\e^{1/2}\right)\right)_+ \\ 
&\times \varphi\left(\mc\left(h+\tilde t-\frac{M\tilde \al^\perp\cdot\tilde\al^\perp}2\right)-j_1+O\left(\e^{1/2}\right), -\frac{|\Phi'_x|}{\e^{1/2}}(\Phi''_{x^\perp y^\perp })^{-1}\tilde \al^\perp-j^\perp+O(1)\right)\\
&+O(\e^2),\quad \tilde t:=t/\e,\ \tilde \al^\perp:=\al^\perp/\e^{1/2}.
\end{split}
\ee

Since $\varphi$ is compactly supported, the number of terms in the sum in \eqref{f-int-Y} is bounded. Using additionally that $|\tilde\al^\perp|\le A<\infty$, the sum itself is bounded as well. Finally, combining with the fact that $|(a+O(\e^{1/2}))_+-a_+|=O(\e^{1/2})$ uniformly in $a\in\br$ and using \eqref{yprp-sol} and \eqref{y1-sol} gives
\be\label{f-int-Y-v2}\begin{split}
&g_\e(Y(\al,t;\xh))\\
&=\frac{f(x_0)b(x_0,y_0)2\pi\e/\mc}{\sqrt{\text{det}N}}
\sum_{j\in r+\mathbb Z^3}\left(j_1+\frac{M_1 \tilde \al^\perp\cdot \tilde \al^\perp}2\right)_+ \\ 
&\times \varphi\left(\mc\left(h+\tilde t-\frac{M\tilde \al^\perp\cdot\tilde\al^\perp}2\right)-j_1+O\left(\e^{1/2}\right), -\frac{|\Phi'_x|}{\e^{1/2}}(\Phi''_{x^\perp y^\perp })^{-1}\tilde \al^\perp-j^\perp+O(1)\right)\\
&+O(\e^{3/2}).
\end{split}
\ee
In view of \eqref{f-int-Y-v2}, denote
\be\label{psi-def}
\psi(q,u):=\sum_{j\in r+\mathbb Z^3}(j\cdot\bt_0+q)_+
\varphi(u-j),\ q\in\br,\ u\in\br^3. 
\ee
Here we will need a higher order approximation of $Y(\al,t;\xh)$ than the one in \eqref{yprp-sol},  \eqref{y1-sol}:
\be\label{Y-appr-next}\begin{split}
Y^\perp/\e&=-|\Phi'_x|(\Phi''_{x^\perp y^\perp })^{-1}\tilde\al^\perp+A_1(\tilde\al^\perp,\tilde\al^\perp) +A_2\tilde t +A_3h+O(\e^{1/2}),\\
Y_1/\e&=\mc\left(h+\tilde t-\frac{M\tilde\al^\perp\cdot\tilde\al^\perp}2\right)+O(\e^{1/2}),
\end{split}
\ee
where $A_1$ is a bilinear map $\br^2\times\br^2\to\br^2$, and $A_2,A_3\in \br^2$. Moreover, the two $O(\e^{1/2})$ terms in \eqref{Y-appr-next} depend smoothly on $\tilde t$ and $\tilde\al^\perp$. In particular, differentiation with respect to $\tilde t$ does not change the order of these terms as $\e\to0$. Hence we rewrite \eqref{f-int-Y-v2} as follows
\be\label{f-int-Y-alt}\begin{split}
&g_\e(Y(\al,t;\xh))=\frac{f_0 b_0 2\pi\e/\mc}{\sqrt{\text{det}N}}
\psi\left(q(\tilde\al^\perp),u(\tilde\al^\perp,\tilde t)\right) + O(\e^{3/2}),\ q(\tilde\al^\perp):=\frac{M_1\tilde\al^\perp\cdot \tilde\al^\perp}2,\\ 
&u(\tilde\al^\perp,\tilde t):=\left(\mc\left(h+\tilde t-\frac{M\tilde \al^\perp\cdot\tilde\al^\perp}2\right)+O\left(\e^{1/2}\right)\right.,\\ 
&\hspace{2cm} \left.-\frac{|\Phi'_x|}{\e^{1/2}}(\Phi''_{x^\perp y^\perp})^{-1}\tilde \al^\perp+A_1(\tilde\al^\perp,\tilde\al^\perp) +A_2\tilde t +A_3h+O(\e^{1/2})\right).
\end{split}
\ee
We also have
\be\label{uder}\begin{split}
\left.\pa u(\tilde\al^\perp,\tilde t)/\pa \tilde t\right|_{\tilde t=0}&=\mc\left(1+O(\e^{1/2}),A_2/\mc +O(\e^{1/2})\right),\\
\left.\pa^2 u(\tilde\al^\perp,\tilde t)/\pa \tilde t^2\right|_{\tilde t=0}&=O(\e^{1/2}).
\end{split}
\ee 
  
%\be\label{f-int-Y-alt}\begin{split}
%g_\e(Y(\al,t;\xh))&\sim\frac{f_0 (\e/c)^{\frac{n-1}2}}{\sqrt{\text{det}M}|S^{n-2}|}\sum_{j\in r+\mathbb Z^3}(\al_0\cdot j)_+^{\frac{n-1}2}\varphi\left((h+\tilde t)v+N\tilde\al^\perp-j\right),\\
%v&=(c,v^\perp),\ N=\begin{pmatrix}0 & 0\\ 0 & \check N \end{pmatrix}.
%\end{split}
%\ee   

\section{Estimating the term $f_{\chi\e}^{(1)}$.}\label{piece-one}

To study $f_{\chi\e}^{(1)}$ we need the following lemma, which follows immediately from \eqref{psi-def} and the properties IK1--IK3 of $\varphi$.

\begin{lemma}\label{psi_asymp} Partial derivatives of $\psi(q,u)$ with respect to $u$ up to the order two are continuous. Also, one has
\begin{equation}\label{psi_pro}
\psi(q,u)=\psi(q+m\cdot\bt_0,u-m),\ \forall m \in \mathbb Z^3,
\end{equation}
and, for some $c>0$,
\begin{equation}\label{psi-zero}
\psi(q,u)=0 \text{ if } u\cdot\bt_0+q<-c;\quad \psi(q,u)=u\cdot\bt_0+q \text{ if }  u\cdot\bt_0+q>c.
\end{equation}
%Also, for $K=0,1$,
%\be\label{psi-ass}
%\left(\frac{\pa}{\pa s}\right)^K\psi(q,s\bt_0+u)=
%\left(\frac{d}{d s}\right)^K (-s)+O\left(|s|^{-K}\right),\ s\to -\infty,
%\ee
%and
%\be\label{psi-ass-n}
%\left(\frac{\pa}{\pa s}\right)^2\psi(q,s\bt_0+u)=O\left(|s|^{-2}\right),\ s\to -\infty.
%\ee
%\begin{equation}\label{psi_u-der}
%\frac{\pa}{\pa q}\psi_1(q,s\bt_0+u),|\nabla_u \psi_1(q,s\bt_0+u)|=O(s^{-\frac{n+1}2}),\ s\to-\infty.
%\end{equation}
% The $O(1/s)$ term in \eqref{psi-ass} can be differentiated with respect to $s$ and drops the order by 1 with each differentiation.
%The asymptotic formulas in \eqref{psi-ass} and \eqref{psi-ass-n} hold uniformly in $q$ and $u$ when the two variables are confined to bounded sets.
\end{lemma}

Denote
\be\label{U-def}
U(q,u):=-\left.\frac{\pa^2}{\pa \tau^2}\psi(q,u+\tau(1,A_2/\mc))\right|_{s=0}.
\ee
%From \eqref{uder},
%\be\label{big-U-expr}
%U(q,u):=-\psi''_{uu}(q,u)(\mc,A_2)\cdot(\mc,A_2)+O(\e^{1/2}).
%\ee
The following lemma is a direct consequence of Lemma~\ref{psi_asymp} (see also properties IK3, IK4 of $\varphi$).
\begin{lemma}\label{Uder-bdd} The function $U(q,u)$ has piecewise continuous bounded first order partial derivatives with respect to $q$ and $u$. Also,
\begin{equation}\label{U_pro}
U(q,u)=U(q+m\cdot\bt_0,u-m),\ \forall m \in \mathbb Z^3,
\end{equation}
and, for some $c>0$,
\begin{equation}\label{U_compsup}
U(q,u)\equiv 0 \text{ if } |u\cdot\bt_0+q|>c.
\end{equation}
\end{lemma}

Note that the derivative in the inversion formula \eqref{inv-formula} is with respect to $t$. Using that $t=\e\tilde t$ (cf. \eqref{f-int-Y}) and taking into account \eqref{uder}, \eqref{U-def}, in the formula below we will acquire the factor $(\e/\mc)^2$. Thus, in terms of $U$, the expression for $f_{\chi\e}^{(1)}$ becomes after changing variables $\al^\perp\to\tilde\al^\perp$ (this brings the factor $\e$), setting $\tilde t=0$, and using \eqref{f-int-Y-alt}, \eqref{uder}, \eqref{U-def}:
\be\label{recon}\begin{split}
f_{\chi\e}^{(1)}(\xh)&=\frac{1}{4\pi^2}\frac{f_0 (2\pi\e/\mc)}{\sqrt{\text{det}N}}\frac{\e}{(\e/\mc)^2}\int_{|\tilde\al^\perp|<A} U\left(q(\tilde\al^\perp),u_0(\tilde\al^\perp)\right)
d\tilde\al^\perp+O(\e^{1/2}),\\
u_0(\tilde\al^\perp):&=\left(\mc\left(h-\frac{M\tilde \al^\perp\cdot\tilde\al^\perp}2\right),
-\frac{|\Phi'_x|}{\e^{1/2}}(\Phi''_{x^\perp y^\perp })^{-1}\tilde \al^\perp+A_1(\tilde\al^\perp,\tilde\al^\perp)+A_3h\right).
\end{split}
\ee
Two simplifications have been made in deriving \eqref{recon}. First, using that first and second order derivatives of $\psi$ are bounded, it follows from \eqref{uder} that
\be\label{uder-cons}
\left.\frac{\pa^2}{\pa\tilde t^2}\psi\left(\cdot,u(\tilde\al^\perp,\tilde t)\right)\right|_{\tilde t=0}=\mc^2\left.\frac{\pa^2}{\pa \tau^2}\psi(\cdot,u(\tilde\al^\perp,0)+\tau(1,A_2/\mc))\right|_{\tau=0}+O(\e^{1/2}).
\ee
Second, since the derivatives of $U$ are bounded, the integral with respect to $\tilde\al^\perp$ is over a bounded set, and the coefficient in front of the integral is bounded, using \eqref{uder-cons} and then replacing $u(\tilde\al^\perp,0)$ with $u_0(\tilde\al^\perp)$ in the arguments of $U$ leads to the term $O(\e^{1/2})$ outside the integral.

Approximate the domain $|\tilde\al^\perp|<A$ by a union of non-overlapping small squares of size $\de$. Let these squares be denoted $B_k$, $k=1,2,\dots,O(\de^{-2})$. By \eqref{recon}, 
\be\label{tilted-plane}\begin{split}
u_0(\tilde\al^\perp)&=u_0(\tilde\al^\perp_k)+\Delta u_k(\tilde \al^\perp)+O(\de),\ \tilde\al^\perp\in B_k,\\
\Delta u_k(\tilde \al^\perp):&=\left(0,-|\Phi'_x|(\Phi''_{x^\perp y^\perp})^{-1}((\tilde \al^\perp-\tilde\al^\perp_k)/\e^{1/2}\right)\in\bt_0^\perp,
\end{split}
\ee
where $\tilde\al^\perp_k$ is the center of $B_k$. By Lemma~\ref{Uder-bdd},  
\be\label{U-appr}\begin{split}
U\left(q(\tilde\al^\perp),u_0(\tilde\al^\perp)\right)&=U\left(q(\tilde\al^\perp_k),u_0(\tilde\al^\perp_k)+\Delta u_k(\tilde \al^\perp)\right)+O(\de),\ \tilde\al^\perp\in B_k.
\end{split}
\ee
Therefore
\be\label{part-recon}\begin{split}
&\int_{B_k} U\left(q(\tilde\al^\perp),u_0(\tilde\al^\perp)\right)
 d\tilde\al^\perp\\
&=\int_{B_k} \left[U\left(q(\tilde\al^\perp_k),u_0(\tilde\al^\perp_k)+\Delta u_k(\tilde \al^\perp)\right)+O(\de)\right]
 d\tilde\al^\perp\\
&=\int_{B_k} \left[U\left(q(\tilde\al^\perp_k)+\lfloor u_0(\tilde\al^\perp_k)+\Delta u_k(\tilde \al^\perp)\rfloor\cdot\bt_0,\{u_0(\tilde\al^\perp_k)+\Delta u_k(\tilde \al^\perp)\}\right)+O(\de)\right]d\tilde\al^\perp\\
&=\int_{B_k} \left[U\left(q(\tilde\al^\perp_k)+u_0(\tilde\al^\perp_k)\cdot\bt_0-\{u_0(\tilde\al^\perp_k)+\Delta u_k(\tilde \al^\perp)\}\cdot\bt_0,\right.\right.\\
&\hspace{5cm}\left.\left.\{u_0(\tilde\al^\perp_k)+\Delta u_k(\tilde \al^\perp)\}\right)+O(\de)\right]d\tilde\al^\perp,
\end{split}
\ee
where we have used that $\Delta u_k(\tilde \al^\perp)\cdot\bt_0=0$. In \eqref{part-recon} and below the fractional part of a vector is computed component wise: $\{u\}=(\{u_1\},\{u_2\},\{u_3\})$, where $\{u_i\}=u_i-\lfloor u_i \rfloor$ and $\lfloor u_i \rfloor$ is the largest integer not exceeding $u_i$.

Pick any $m\in\mathbb Z^3$, $m\not=0$. Let $m^\perp$ be the projection of $m$ onto the plane $\bt_0^\perp$. Condition LG2 in Definition~\ref{gponS} implies that $m^\perp\not=0$. By the local Bolker condition DF4, $\Phi''_{x^\perp y^\perp}$ is non-degenerate, so the vector $(\Phi''_{x^\perp y^\perp})^{-T}m^\perp\in\br^2$ is not zero. Using \eqref{tilted-plane}, the standard Weyl-type argument (cf. \cite{KN_06}) implies that
\be\label{recon-lim}\begin{split}
&\lim_{\e\to0}\int_{B_k} U\left(q(\tilde\al^\perp),u_0(\tilde\al^\perp)\right)
 d\tilde\al^\perp\\
&=\left(\int_{[0,1]^3} U\left(q(\tilde\al^\perp_k)+u_0(\tilde\al^\perp_k)\cdot\bt_0-\om\cdot\bt_0,\om\right)d\om+O(\de)\right)\text{Vol}(B_k).
\end{split}
\ee
Indeed, consider the function $U_1(q,\om):=U(q-\om\cdot\bt_0,\om)$. Clearly, $U_1(q,\om)$ is periodic: $U_1(q,\om)=U_1(q,\om+m)$, $m\in\mathbb Z^3$. Expand $U_1(q,\om)$ in a Fourier series: 
\be\label{fourser}
U_1(q,\om)=\sum_{m\in \mathbb Z^3} A_m(q)\exp(2\pi i m\cdot\om).
\ee
By the argument preceding \eqref{recon-lim},
\be\label{mdp}\begin{split}
m\cdot\left(0,-|\Phi'_x|(\Phi''_{x^\perp y^\perp})^{-1}(\tilde \al^\perp-\tilde\al^\perp_k)\right)&=-\frac1{|\Phi'_x|}\left[(\Phi''_{x^\perp y^\perp})^{-T}m^\perp\right]\cdot(\tilde \al^\perp-\tilde\al^\perp_k),\\ (\Phi''_{x^\perp y^\perp})^{-T}m^\perp&\not=0,|m|>0.
\end{split}
\ee
Therefore, 
\be\label{part-recon-2}\begin{split}
&\lim_{\e\to0}\left|\int_{B_k} 
\exp\left(2\pi i m\cdot
\left(u_0(\tilde\al^\perp_k)+\Delta u_k(\tilde \al^\perp)\right)\right)d\al^\perp\right|\\
&=\lim_{\e\to0}\left|\int_{B_k} 
\exp\left(
-\frac{2\pi i }{|\Phi'_x|\e}\left[(\Phi''_{x^\perp y^\perp})^{-T}m^\perp\right]\cdot\tilde \al^\perp\right)d\tilde\al^\perp\right|=0,\ 
|m|>0.
\end{split}
\ee
In other words, only the term corresponding to $m=0$ survives, and the desired assertion follows from the standard approximation argument.

By \eqref{gej}, \eqref{f-int-Y-alt} and \eqref{recon},
\be\label{sum_qu}
q(\tilde\al^\perp_k)+u_0(\tilde\al^\perp_k)\cdot\bt_0=\mc h-\frac{\mc N^{-1}\tilde\al_k^\perp\cdot \tilde\al_k^\perp}2.
\ee
Add the integrals over all the squares $B_k$ and use \eqref{recon}, \eqref{recon-lim}, and \eqref{sum_qu}:
\be\label{recon-v2}\begin{split}
&\lim_{\e\to0} f_{\chi\e}^{(1)}(\xh)\\
&=\kappa\sum_k \int_{B_k}U\left(q(\tilde\al^\perp),u_0(\tilde\al^\perp)\right)
 d\tilde\al^\perp+O(\de)\\
 &=\kappa\sum_k\left(\int_{[0,1]^3} U\left(\mc h-\frac{\mc N^{-1}\tilde\al_k^\perp\cdot \tilde\al_k^\perp}2-\om\cdot\bt_0,\om\right)d\om+O(\de)\right)\text{Vol}(B_k)+O(\de),\\
&\kappa:=\frac{\mc}{2\pi}\frac{f_0}{\sqrt{\text{det}N}}.
\end{split}
\ee
Since $\de>0$ can be as small as we like, \eqref{recon-v2} implies
\be\label{recon-v3}\begin{split}
\lim_{\e\to0} & f_{\chi\e}^{(1)}(\xh)
=\kappa\int_{|\tilde\al^\perp|<A}\int_{[0,1]^3} U\left(\mc h-\frac{\mc N^{-1}\tilde\al^\perp\cdot \tilde\al^\perp}2-\om\cdot\bt_0,\om\right)d\om\,d\tilde\al^\perp.
\end{split}
\ee
The second argument of $U$ is bounded and $N$ is negative definite, so by \eqref{U_compsup} the integral with respect to $\tilde\al^\perp$ over the set $|\tilde\al^\perp|<A$, when $A>0$ is large enough (but fixed), can be replaced by the integral over all $\br^2$. Changing variables and integrating in spherical coordinates gives
\be\label{recon-v4}\begin{split}
\lim_{\e\to0}  f_{\chi\e}^{(1)}(\xh)
&=\kappa\int_{\br^2}\int_{[0,1]^3} U\left(\mc h-\frac{\mc N^{-1}\tilde\al^\perp\cdot \tilde\al^\perp}2-\om\cdot\bt_0,\om\right)d\om\,d\tilde\al^\perp\\
&=\kappa\frac{2\sqrt{\det N}}{\mc}\int_{\br^2} \int_{[0,1]^3}U\left(\mc h+v\cdot v-\om\cdot\bt_0,\om\right) d\om d v\\
&=f_0 \ioi \int_{[0,1]^3}U\left(\mc h+\tau-\om\cdot\bt_0,\om\right) d\om\, d\tau.
\end{split}
\ee

The integral with respect to $\om$ can be evaluated explicitly. Recall that in our coordinates, $\bt_0=(1,0,0)$ (cf. \eqref{bt0def}). From \eqref{psi-def} and \eqref{U-def},
\be\label{Uuint}\begin{split}
-\int_{[0,1]^3}&U(q-\om\cdot\bt_0,\om) d\om\\
&=\int_{[0,1]^3}\left.\frac{\pa^2}{\pa \tau^2}\psi(q-\om\cdot\bt_0,\tau(1,A_2/\mc)+\om)\right|_{\tau=0}d\om\\
&=\left.\frac{\pa^2}{\pa \tau^2}\int_{[0,1]^3}\sum_{j\in r+\mathbb Z^3}(q-\bt_0\cdot (\om-j))_+
\varphi(\tau(1,A_2/\mc)+\om-j)d\om\right|_{\tau=0}\\
&=\left.\frac{\pa^2}{\pa \tau^2}\int_{\br^3}(q-\bt_0\cdot \om)_+
\varphi(\tau(1,A_2/\mc)+\om)d\om\right|_{\tau=0}\\
&=\left.\frac{\pa^2}{\pa \tau^2}\int_{\br^3}(q-\bt_0\cdot \om)_+
\varphi(\tau(1,0)+\om)d\om\right|_{\tau=0}=\hat \varphi(\bt_0,q).
\end{split}
\ee
Substitute \eqref{Uuint} into \eqref{recon-v4} 
\be\label{recon-v5}
\lim_{\e\to0}  f_{\chi\e}^{(1)}(\xh)
=-f_0 \ioi \hat \varphi(\bt_0,\mc h+\tau) d \tau=f_0 \left(-\int^{\infty}_{\mc h} \hat \varphi(\bt_0,\tau) d \tau\right).
\ee
Since $\varphi$ is normalized and compactly supported, 
\begin{equation}\label{U-lims}
-\int^{\infty}_{\mc h} \hat \varphi(\bt_0,\tau) d\tau\to\begin{cases}0,& h>c, \\
-1,& h<-c.\end{cases}
\end{equation}
for some $c>0$.

Using the definition \eqref{crt-def} and some simple transformations, we can rewrite the integral in \eqref{recon-v5} in two different forms
\be\label{phi-int-alt}
\int^{\infty}_{\mc h} \hat \varphi(\bt_0,\tau) d\tau=\int^{\infty}_{|\Phi'_x| h} \hat \varphi(\Phi'_y,\tau) d\tau=\int_{\Phi'_x \tilde x +\Phi'_y\tilde y>0} \varphi(\tilde y) d\tilde y.
\ee

\section{Analysis of the term $f_{\chi\e}^{(2)}$.}\label{lot}

\begin{lemma}\label{tough-cookie} One can find $\omega>0$, $\e_0>0$ small enough and $A>0$ large enough so that $g(y)$ is smooth in a neighborhood of all $y$ such that $(y-Y(\al,0;\xh))/\e\in\text{supp}(\varphi)$ for any $\al\in\Omega_2,\tilde x\in\tilde U$ (cf. \eqref{rec-pt}), and $0<\e<\e_0$. 
\end{lemma}

\begin{proof} 
Fix some $c>0$ sufficiently large. Using that $N(x_0)$ is negative definite, equation \eqref{pzsum} implies that we can find $A>0$ large enough and $\omega>0$ small enough so that $P(Y(\al,0;\xh))>c\e$ for all $\al\in\Omega_2$ and all $\tilde x\in\tilde U$ provided that $\e$ is small enough. Since $\varphi$ is compactly supported and $P(y)$ is smooth, $P(y)>0$ for all $y$ such that $(y-Y(\al,0;\xh))/\e\in\text{supp}(\varphi)$ (this is where we use that $c>0$ is sufficiently large). Therefore, $g$ is a smooth function in a neighborhood of all such $y$ because in this case we can drop the subscript $'+'$ from $P_+(y)$ in \eqref{GRT-loc-v3}.
\end{proof}

Lemma~\ref{tough-cookie} implies that the inversion formula \eqref{inv-formula} and its discrete analogue do not see singularities in the data when $\al\in \Omega_2$ provided that $\omega,A$ are selected as in the proof of Lemma~\ref{tough-cookie}. Since $|\Omega_1|=O(\e)$, clearly the limit
$\lim_{\e\to0^+}f_{\chi\e}^{(2)}(\xh)$ exists and is independent of $\tilde x\in\tilde\us$. 

Using again that $N(x_0)$ is negative definite, there exist sufficiently small neighborhoods $\us_1$ of $x_0$ and $\vs_1$ of $y_0$ such that $\s_y\cap\us_1$ is on the exterior side of $\s$ whenever $\s_y$ is tangent to $\s$ and $y\in\vs_1$. This implies that if $x\in\us_1$ and $x$ is on the interior side of $\s$, then there is no $y\in\vs_1$ such that $\s_y$ contains $x$ and is tangent to $\s$. In turn, this implies that the data $g(y)$, $y=Y(\al,t=0;x)\in\vs_1$, which is used to compute $f_\chi(x)$ is also smooth. Then, clearly, $\lim_{\e\to0^+}(f_{\chi\e}^{(2)}(x_0+\e h\al_0)-f_{\chi}(x_0+\e h\al_0))=0$ for any $h>0$, and
\be\label{f2f3}
\lim_{\e\to0^+}f_{\chi\e}^{(2)}(\xh)=f_{\chi}(x_{0^+}).
\ee
This concludes the proof of the first part of the theorem.

\section{Contribution of remote singularities}\label{remote}

Suppose $\s_{y_0}$ is tangent to $\s$ at some $z_0\in \s_{y_0}$, $z_0\not=x_0$, $\s$ is smooth at $z_0$, and $N(z_0)$ is either positive definite or negative definite. Set $\al_0=\Phi'_x(z_0,y_0)/|\Phi'_x(z_0,y_0)|$, so that $y_0=Y(\al_0,0;x_0)$. As before, $\vs_1$ is a small neighborhood of $y_0$, and $\text{supp}(\chi)\subset\vs_1$.

As follows from assertion (1) of Lemma~\ref{ssg-smooth} (with $x_0$ replaced by $z_0$ as the point of tangency), the set of $y\in\vs_1$ such that $\s_y$ is tangent to $\s$ near $z_0$ is a smooth submanifold of $\vs_1$ through $y_0$, and the vector $\Phi'_y(z_0,y_0)$ is normal to it at $y_0$. The local equation of the manifold is $P(y)=0$, where the function $P$ is the same as in Lemma~\ref{ssg-smooth}. By assertion (2) of the lemma, $P'_y(y_0)\not=0$. 

Additionally, $\T_{x_0}$ is another smooth submanifold through $y_0$, and $\Phi'_y(x_0,y_0)$ is normal to it at $y_0$. By the assumption DF3 of no conjugate points, $\Phi'_y(x_0,y_0)$ and $\Phi'_y(z_0,y_0)$ are not parallel, so the intersection of the two submanifolds is a smooth curve $\Gamma_{x_0}$ through $y_0$. This is the same curve $\Gamma_{x_0}$, which was introduced in \eqref{ga-def}. From this argument it is easy to see that  $\Gamma_{x}$ depends smoothly on $x$ near $x=x_0$.

\begin{theorem}\label{thm:remote} Pick a globally generic pair $(x_0,y_0)\in\ir$ such that $x_0\not\in \s$ and $\s_{y_0}$ is tangent to $\s$ at $z_0\in \s$. Suppose $\text{supp}(f)$ is contained in a sufficiently small neighborhood of $z_0$. One has 
\be\label{res-remote}
\lim_{\e\to0}f_{\chi\e}(\xh)=f_{\chi}(x_0).
\ee
\end{theorem}

\begin{proof} Suppose first that the reconstruction point is $x_0$. Since the reconstruction point is fixed, the dependence of various quantities on $x_0$ is omitted from notations in most places when there is no risk of confusion. Then
\be\label{inv-formula-p1}\begin{split}
f_{\chi\e}(x_0)=&\int_{S_+^2}B(\al) \sum_j
g(\e j)\left.\left(\frac{\pa}{\pa t}\right)^2 \ik\left(\frac{Y(\al,t)-\e j}\e\right)\right|_{t=0} d\al,\\
B(\al):=&-\frac{\chi(Y(\al,0))}{4\pi^2}\frac{1} {b(x_0,Y(\al,0))}.
\end{split}
\ee
By Lemma~\ref{der-g}, 
%\be\label{GRT-loc-v2-parts}
%\begin{split}
%g^{(1)}(Y(\al,t;x_0)):=&(t-t_0(\al,x_0))_+G_0(\al),\ G_0(\al):=G(\al,t_0(\al,x_0);x_0),\\
%g^{(2)}(Y(\al,t;x_0)):=&(t-t_0(\al,x_0))_+(G(\al,t;x_0)-G_0(\al)).
%\end{split}
%\ee
\be\label{GRT-loc-v2-parts}
g(y):=P_+(y)G(y),\ G(y):=G(Z(y),P(y)),
\ee
%Here $z(y)$ and $p(y)$ are smooth functions described in statement (2) of Lemma~\ref{ssg-smooth}. Substitute $g^{(1)}$ into the inversion formula.
%\be\label{inv-formula-p1}\begin{split}
%f_{\chi}^{(1)}(x_0)=&\int_{S^2}b(\al) \sum_j
%(\ot(\e j)-t_0(\oal(\e j)))_+G_0(\oal(\e j))\\
%&\times\left.\left(\frac{\pa}{\pa t}\right)^2 \ik\left(\frac{Y(\al,t)-\e j}\e\right)\right|_{t=0} d\al,\
%b(\al):=-\frac{\chi(\al)}{8\pi^2}\frac{w(x_0,\al)} {b(x_0,Y(\al,0))}.
%\end{split}
%\ee
%\be\label{inv-formula-p1}\begin{split}
%f_\e^{(1)}(x_0):=&\int_{S_+^2}B(\al) \sum_j
%g^{(1)}(\e j)\left.\left(\frac{\pa}{\pa t}\right)^2 \ik\left(\frac{Y(\al,t)-\e j}\e\right)\right|_{t=0} d\al,\\
%B(\al):=&-\frac{\chi(Y(\al,0))}{4\pi^2}\frac{1} {b(x_0,Y(\al,0))}.
%\end{split}
%\ee
%\textcolor{red}{By construction, $t\equiv\ot(Y(\al,t))$ and $\al\equiv\oal(Y(\al,t))$, so
%\be\label{prop-ders}
%\oal'_y Y'_t\equiv0,\ \ot'_y\cdot Y'_t\equiv 1.
%\ee}
Since $\ik$ is compactly supported, we can expand the factor $P(y)$ in \eqref{GRT-loc-v2-parts} in the Taylor series centered at $Y(\al,0)$. Let $L(y)$ be its linear term:
\be\label{lin-term}
L(y):=P(Y(\al,0))+P'_y(Y(\al,0))\cdot (y-Y(\al,0)).
\ee
%write
%\be\label{tt0-appr}\begin{split}
%t(\e j)&=\left.\frac{dt(y)}{dy}\right|_{y=Y(\al,0)}(\e j-Y(\al,0))+O(\e^2),\\
%t_0(\e j)&=t_0(\al)+\frac{dt_0(\al)}{d\al}\left.\frac{d\al(y)}{dy}\right|_{y=Y(\al,0)}(\e j-Y(\al,0))+O(\e^2).
%\end{split}
%\ee
We begin by looking at the expression, which is obtained by ignoring the second and higher order terms in the expansion of $P$:
\be\label{inv-formula-p2}\begin{split}
J_\e^{(1)}:=&\int_{S_+^2}B(\al) \sum_j G(\e j)L_+(\e j)
\left.\left(\frac{\pa}{\pa t}\right)^2 \ik\left(\frac{Y(\al,t)-\e j}\e\right)\right|_{t=0} d\al.
\end{split}
\ee
Clearly,
\be\label{ikderive}\begin{split}
&\left.\left(\frac{\pa}{\pa t}\right)^2  \ik\left(\frac{Y(\al,t)-\e j}\e\right)\right|_{t=0} d\al\\
&\quad=\frac1{\e^2}\nabla_{u(\al)}^2\ik\left(\frac{Y(\al,0)}\e-j\right)
+\frac1{\e}\ik'_y\left(\frac{Y(\al,0)}\e-j\right)\cdot Y_{tt}''(\al,0),\\ 
u(\al)&:=Y'_t(\al,0),\ \nabla_{u}^2\ik(y)=\left.\left(\frac{\pa}{\pa t}\right)^2 \ik(y+tu)\right|_{t=0}.
\end{split}
\ee
Consider the most singular part of $J_\e^{(1)}$, which is obtained by using the first term on the right in \eqref{ikderive} and replacing $G(\e j)$ with $G(Y(\al,0))$:
\be\label{jlst}\begin{split}
J_\e^{(1a)}:=&\int_{S_+^2}B_1(\al) \sum_j L_+(\e j) \frac1{\e^2}\nabla_{u(\al)}^2\ik\left(\frac{Y(\al,0)}\e-j\right) d\al,\\ B_1(\al):=&B(\al)G(Y(\al,0)).
\end{split}
\ee
%Here we have used that $G_0(\al(\e j))=G_0(\al)+O(\e)$. This allowed us to replace $G_0(\al(\e j))$ with $G_0(\al)$ and absorb the latter by the function $b(\al)$.
In view of \eqref{lin-term} and \eqref{jlst}, similarly to \eqref{psi-def} and \eqref{U-def}, introduce the function
\be\label{psi-new}
\psi(q,v;\al):=\sum_j
\left(e(\al)\cdot(j-v)+q\right)_+ \nabla_{u(\al)}^2\ik\left(v-j\right),\ e(\al):=P'_y(Y(\al,0)).
\ee
Clearly, 
\begin{enumerate}
\item\label{psi1} $\psi$ is compactly supported in $q$ (by \eqref{ker-int}),
\item\label{psi2} $\psi$ has bounded first order partial derivatives, and 
\item\label{psi3} $\psi(q,v;\al)=\psi(q,v-m;\al)$ for any $m\in\mathbb Z^3$. 
\end{enumerate}
Using \eqref{psi-new} in \eqref{jlst} yields:
\be\label{jlst-v2}\begin{split}
J_\e^{(1a)}=&\frac1{\e}\int_{S_+^2}B_1(\al) \psi\left(\frac{P(Y(\al,0))}{\e},\frac{Y(\al,0)}\e;\al\right) d\al.
\end{split}
\ee

Introduce local coordinates $s=(s_1,s_2)$ on $S_+^2$ so that $s_1\equiv P(Y(\al,0))$ in a neighborhood of $\al_0$. As is shown at the beginning of this section, $\Gamma$ is the transverse intersection of the submanifolds $\T_{x_0}$ and $\{y\in \text{supp}(\chi):\,P(y)=0\}$. By the first equation in \eqref{Y-prp}, $\al\to Y(\al,0)\in\T_{x_0}$ is a regular parametrization near $\al_0$. In \eqref{Y-prp}, $y_0=Y(\al_0,0;x_0)$, and $Y^\perp$ is determined by the projection onto the plane $(\Phi_y'(x_0,y_0))^\perp$ (as opposed to $(\Phi_y'(z_0,y_0))^\perp$). Hence $P'_y(y_0)\not=0$ (cf. \eqref{pprime}) implies $\pa P(Y(\al,0))/\pa\al^\perp\not=0$ near $\al_0$. Therefore the preimage of $\Gamma\cap\vs_1$, given by $\{\al\in S_+^2:\,P(Y(\al,0))=0,\,Y(\al,0)\in\vs_1\}$ is also a smooth curve, and local coordinates $(s_1,s_2)$ with the required property do exist. Then 
\be\label{jlst-v3}\begin{split}
J_\e^{(1a)}=&\frac1{\e}\int_{\br^2} B_1(\al(s)) \psi\left(\frac{P(Y(\al(s),0))}{\e},\frac{Y(\al(s),0)}\e;\al(s)\right) \left|\frac{\pa \al}{\pa s}\right|ds+O(\e)\\
=&\int_{\br^2} B_2(s_2) \psi\biggl(\tilde s_1,\frac{Y(\al(0,s_2),0)}\e+\left.\frac{\pa Y(\al(s_1,s_2),0)}{\pa s_1}\right|_{s_1=0}\tilde s_1;\al(0,s_2)\biggr) d\tilde s_1 ds_2\\
&\hspace{3cm}
+O(\e),\\
&B_2(s_2):=B_1(\al(0,s_2))\left|\frac{\pa \al(s_1=0,s_2)}{\pa s}\right|,\
\tilde s_1=s_1/\e.
\end{split}
\ee
In the first line, the integral is over the bounded set $\{s\in\br^2:\,Y(\al(s),0)\in\text{supp}(\chi)\}$. In the second line, the integral can be confined to a bounded set $\{(\tilde s_1,s_2)\in\br^2:\, |\tilde s_1|<\tilde A,Y(\al(0,s_2))\in \text{supp}(\chi)\}$ for some $\tilde A>0$ large enough.

\begin{lemma}\label{key-ass-res} Let $D$ be a rectangle $D:=[a_1,b_1]\times [a_2,b_2]$. 
Consider a function $\psi\in C(D\times\br^3)$. Suppose 
%\begin{enumerate}
%\item $\psi$ is compactly supported in $s_{1,2}$: $\psi(s_1,s_2,y)\equiv0$ if $|s_1|+|s_2|>A$ for some $A>0$, and
%\item 
$\psi$ is periodic: $\psi(s,y)=\psi(s,y+m)$ for any $m\in\mathbb Z^3$ and $(s,y)\in D\times\br^3$.
%\end{enumerate}
Let $Y:[a_2,b_2]\to\br^3$ be a $C^1$ function with the following property. For any $m\in\mathbb Z^3$, $|m|>0$, 
\begin{enumerate}
\item The set $\{s_2\in[a_2,b_2]:\, |m\cdot Y'(s_2)|\le\de\}$ is contained in a finite number of intervals for all $\de>0$ sufficiently small (this number may depend on $m$ and $\de$), and 
\item The sum of the lengths of these intervals goes to zero as $\de\to0$.  
\end{enumerate}
Then one has
\be\label{limiint}\begin{split}
\lim_{\e\to0^+}\int_D\psi\left(s,\frac{Y(s_2)}{\e}\right)ds
=\int_D\int_{[0,1]^3}\psi(s,y)dyds.
\end{split}
\ee
\end{lemma}
\begin{proof} Pick any $\de_1>0$. Let $\Psi_m(s)$ denote the coefficients of the Fourier expansion of $\psi(s,y)$ with respect to $y$. We can find $M>0$ large enough and a partition of $D$ into sufficiently small rectangles such that 
\be\label{trigpol}\begin{split}
\sup_{s,y}\left|\psi(s,y)-\sum_{|m|\le M}\tilde\Psi_m(s)\exp\left(2\pi i m\cdot y\right)\right|\le\de_1.
\end{split}
\ee
Here $\tilde\Psi_m$ is an approximation of $\Psi_m$, which is constant on each rectangle of the partition. Thus, the lemma will be proven if we show that
\be\label{keyres}
\int_a^b \exp\left(2\pi i m\cdot Y(s_2)/\e\right)ds_2 \to 0,\ \e\to0,\text{ for any }[a,b]\subset [a_2,b_2],\ m\in\mathbb Z^3,|m|>0.
\ee
Using assumptions (1) and (2) of the lemma, partition $[a,b]$ into a finite collection of non-overlapping intervals so that (i) their union is as close to $[a,b]$ as we like, and (ii) in each of these intervals $m\cdot Y'(s_2)$ is bounded away from zero. The result now follows immediately. 
\end{proof}

Clearly, Condition GG2 in Definition~\ref{gpnotonS} is independent of the choice of the vector field $\dot\Gamma$ as long as it does not vanish at any point of $\Gamma$. In the $s$-coordinates, $Y(\al(0,s_2),0)$ is a regular parametrization of $\Gamma\cap\vs_1$ because $|\pa Y^\perp/\pa\al^\perp|_{\al=\al_0}\not=0$ and $|\pa\al^\perp/\pa s|\not=0$ (see also the argument following \eqref{jlst-v2}). The latter determinant is computed at $s$ such that $\al(s)=\al_0$. 
Therefore $\pa Y/\pa s_2$ never vanishes on $\Gamma\cap\vs_1$. Condition GG2 implies that $Y(\al(0,s_2))$ satisfies Conditions (1), (2) in Lemma~\ref{key-ass-res}. Set
\be\label{psi1def}
\psi_1(s,y):=B_2(s_2) \psi\biggl(\tilde s_1,y+\left.\frac{\pa Y(\al(s_1,s_2),0)}{\pa s_1}\right|_{s_1=0}\tilde s_1;\al(0,s_2)\biggr),\ s=(\tilde s_1,s_2).
\ee
Using the properties (\ref{psi1})--(\ref{psi3}) of $\psi$, we see that Lemma~\ref{key-ass-res} applies to $\psi_1$. Also, $\psi_1$ is compactly supported. Compact support along $\tilde s_1$ is due to the property (\ref{psi1}) of $\psi$, and along $s_2$ -- due to the cut-off $\chi$. Substituting $\psi_1$ into \eqref{jlst-v3}, using Lemma~\ref{key-ass-res}, and then expressing $\psi_1$ in term of $\psi$ yields
\be\label{jlst1a-v2}\begin{split}
\lim_{\e\to0}J_\e^{(1a)}=
\int_{\br^2} B_2(s_2)\left[\int_{[0,1]^3}\psi\left(\tilde s_1,v;\al(0,s_2)\right)dv\right] d\tilde s_1 ds_2.
\end{split}
\ee
By \eqref{psi-new}, similarly to \eqref{Uuint},
\be\label{ints1}\begin{split}
&\int_{[0,1]^3} \sum_j
\left(e\cdot(j-v)+q\right)_+ \nabla_u^2\ik\left(v-j\right)dv\\
&=\int_{\br^3} 
\left(-e\cdot v+q\right)_+ \nabla_u^2\ik(v)dv
={(e\cdot u)^2}\hat\ik(e,q).
\end{split}
\ee
As $|e|$ not necessarily equals one, \eqref{ints1} assumes the extended definition of the CRT, cf. \eqref{crt-def}. By \eqref{ikderive} and \eqref{psi-new}, 
\be\label{edotu}
e(\al)\cdot u(\al)=\left.\pa P(Y(\al,t))/\pa t\right|_{t=0}=:P'_t(\al),\ \al=\al(0,s_2).
\ee
%Recall that $|e(\al)|\not=0$ since $p'_y(y_0)\not=0$. 
With $\ik$ normalized, using \eqref{ints1} with $q=\tilde s_1$ and \eqref{edotu} in \eqref{jlst1a-v2} gives
\be\label{jlst1a-v3}\begin{split}
\lim_{\e\to0}J_\e^{(1a)}&=
\int_{\br} B_2(s_2)(P'_t(\al(0,s_2)))^2\int\hat\ik(e(\al(0,s_2)),\tilde s_1)d\tilde s_1  ds_2\\
&=\int_{\br} B_2(s_2)(P'_t(\al(0,s_2)))^2ds_2.
\end{split}
\ee
Consequently, from \eqref{jlst-v3} we get
\be\label{jlst1a-v4}\begin{split}
\lim_{\e\to0}J_\e^{(1a)}=&\int_{\br} B_1(\al(0,s_2))\left|\frac{\pa \al(0,s_2)}{\pa s}\right|(P'_t(\al(0,s_2)))^2 ds_2\\
=&\int_{\br^2} B_1(\al(s))(P'_t(\al(s)))^2 \delta(P(Y(\al(s),0)))\left|\frac{\pa \al(s)}{\pa s}\right|ds\\
=&\int_{S_+^2} B_1(\al) \left(\left.\frac{\pa P(Y(\al,t))}{\pa t}\right|_{t=0}\right)^2 \delta(P(Y(\al,0)))d\al.
\end{split}
\ee
In the second line we used that $s=(s_1,s_2)$ and $s_1\equiv P(Y(\al(s),0))$.

Next, consider the second part of $J_\e^{(1)}$, which is obtained by using the second term on the right in \eqref{ikderive} and replacing $G(\e j)$ with $G(Y(\al,0))$:
\be\label{jlst-pb}\begin{split}
J_\e^{(1b)}:=&\int_{S_+^2}B_1(\al) \sum_j 
L_+(\e j) \frac1{\e}\ik'_y\left(\frac{Y(\al,0)}\e-j\right)\cdot Y''_{tt}(\al,0) d\al.
\end{split}
\ee
The function $L_+(y)$ has bounded first derivatives, hence the limit of $J_\e^{(1b)}$ can be easily found:
\be\label{jlst-pb-lim}\begin{split}
\lim_{\e\to0}J_\e^{(1b)}=&\lim_{\e\to0}\int_{S_+^2}B_1(\al) \sum_j
L_+(\e j) \frac1{\e}\ik'_y\left(\frac{Y(\al,0)-\e j}\e\right)\cdot Y''_{tt}(\al,0) d\al\\
=&\lim_{\e\to0}\int_{S_+^2}B_1(\al) \left.\pa_y\sum_j
L_+(\e j) \ik\left(\frac{Y(\al,0)+y-\e j}\e\right)\right|_{y=0}\cdot Y''_{tt}(\al,0) d\al\\
=&\int_{S_+^2}B_1(\al) \left.\pa_y
L_+(Y(\al,0)+y) \right|_{y=0}\cdot Y''_{tt}(\al,0) d\al\\
=&\int_{S_+^2}B_1(\al) P'_y(Y(\al,0))\cdot Y''_{tt}(\al,0) \theta(P(Y(\al,0)))d\al.
\end{split}
\ee
Recall that $P>0$ on the interior side of $\s$.

The final piece of $J_\e^{(1)}$ is
\be\label{jlst-pc}\begin{split}
J_\e^{(1c)}:=&\int_{S_+^2}B(\al) \sum_j (G(\e j)-G(Y(\al,0))) 
L_+(\e j)\\
&\hspace{2cm}\times \left.\left(\frac{\pa}{\pa t}\right)^2 \ik\left(\frac{Y(\al,t)-\e j}\e\right)\right|_{t=0} d\al.
\end{split}
\ee
In an $O(\e)$ neighborhood of $\Gamma$, we have $|L(y)|=O(\e)$. 
%Using the expansion
%\be\label{Gexp}
%G(\e j)-G(Y(\al,0))=G'_y(\e j-Y(\al,0))+\frac{G''_{yy}(\e j-Y(\al,0))\cdot(\e j-Y(\al,0))}2+O(\e^3),
%\ee
For any fixed $y\in\vs_1$, $P(y)>0$, we compute by dropping the subscript `$+$' from $L$:
\be\label{smpl}\begin{split}
&\sum_j (G(\e j)-G(Y(\al,0))) 
L(\e j)  \left.\left(\frac{\pa}{\pa t}\right)^2 \ik\left(\frac{Y(\al,t)-\e j}\e\right)\right|_{t=0}\\
&=\left.\left(\frac{\pa}{\pa t}\right)^2 \left[(G(Y(\al,t))-G(Y(\al,0)))L(Y(\al,t))\right]\right|_{t=0}+O(\e).
\end{split}
\ee
Substitution into \eqref{jlst-pc} gives
\be\label{jlst-pc-lim}\begin{split}
&\lim_{\e\to0}J_\e^{(1c)}\\
&=\int_{S_+^2}B(\al) \left.\left(\frac{\pa}{\pa t}\right)^2 \left[(G(Y(\al,t))-G(Y(\al,0)))L(Y(\al,t))\right]\right|_{t=0} \theta(P(Y(\al,0)))d\al.
\end{split}
\ee
We can apply the limit as $\e\to0$ inside the integral in \eqref{jlst-pc} to obtain \eqref{jlst-pc-lim} because the integrand is uniformly bounded. This follows because the function $(G(y)-G(Y(\al,0)))L(y)$ is smooth away from an $O(\e)$ neighborhood of $\Gamma$, and the integrand is $O(1)$ within that neighborhood.

The final term to be considered arises because of the difference between $P(y)$ (cf. \eqref{GRT-loc-v2-parts}) and its linear approximation $L(y)$ (cf. \eqref{lin-term}):
\be\label{jlst-2a}\begin{split}
J_\e^{(2)}:=&\int_{S_+^2}B(\al) \sum_j G(\e j)
\left[P_+(\e j) - L_+(\e j)\right]\left.\left(\frac{\pa}{\pa t}\right)^2 \ik\left(\frac{Y(\al,t)-\e j}\e\right)\right|_{t=0} d\al.
\end{split}
\ee
%\be\label{jlst-2b}\begin{split}
%J_\e^{(2b)}:=&\int_{S_+^2}B(\al) \sum_j G(\e j)
%\left[(p(\e j))_+ - \left(L(\e j)\right)_+\right]\\ &\times\frac1{\e}\ik'_y\left(\frac{Y(\al,0)}\e-j\right)\cdot Y''_{tt}(\al,0) d\al,
%\end{split}
%\ee
In the domain where  $P(\e j)$ and $L(\e j)$ are both positive, we have 
\be\label{scndtrms}
P_+(\e j)-L_+(\e j)=\frac12 P''_{yy}(Y(\al,0))(\e j-Y(\al,0))\cdot(\e j-Y(\al,0))+O(\e^3).
\ee
This difference is zero if $P(\e j)$ and $L(\e j)$ are both negative. Thus,
\be\label{jlst-2a-st1}\begin{split}
\lim_{\e\to0}&\sum_j G(\e j)
\left[P(\e j) - L(\e j)\right] \left.\left(\frac{\pa}{\pa t}\right)^2 \ik\left(\frac{Y(\al,t)-\e j}\e\right)\right|_{t=0}\\
&=G(Y(\al,0))P''_{yy}(Y(\al,0))Y'_t(\al,0)\cdot Y'_t(\al,0),\ P(Y(\al,0))>0.
\end{split}
\ee
In the region where $P(Y(\al,0))<0$, the limit is obviously zero. Hence
\be\label{jlst-2a-st2}\begin{split}
\lim_{\e\to0}J_\e^{(2)}=\int_{S_+^2}B_1(\al)P''_{yy}(Y(\al,0))Y'_t(\al,0)\cdot Y'_t(\al,0)\theta(P(Y(\al,0))) d\al.
\end{split}
\ee
As before, we can apply the limit as $\e\to0$ inside the integral in \eqref{jlst-2a} to obtain \eqref{jlst-2a-st2} because the integrand is uniformly bounded. Indeed,
\be\label{diffsize}
P_+(y)-L_+(y)=(L(y)+O(\e^2))_+-L_+(y)=O(\e^2),
\ee
so the integrand in \eqref{jlst-2a} remains bounded as $\e\to0$. The domain where $P(y)$ and $L(y)$ are of different signs is a shrinking $O(\e)$ neighborhood of $\Gamma$, and the desired result follows.
%Combine \eqref{jlst-pb-lim} and \eqref{jlst-2a-st2}. Since 
%\be\label{J-comb}\begin{split}
%&e(\al)\cdot Y''_{tt}(\al,0)+p''_{yy}(Y(\al,0)) u(\al)\cdot u(\al)\\
%&=p'_y(Y(\al,0))\cdot Y''_{tt}(\al,0)+p''_{yy}(Y(\al,0)) Y'_{t}(\al,0)\cdot Y'_{t}(\al,0)\\
%&=\left.\frac{\pa^2}{\pa t^2} p(Y(\al,t))\right|_{t=0}=\left.\frac{\pa^2}{\pa t^2} \left(\ot(Y(\al,t))-t_0(\oal(Y(\al,t)))\right)\right|_{t=0}\\
%&=\left.\frac{\pa^2}{\pa t^2} \left(t-t_0(\al)\right)\right|_{t=0}=0,
%\end{split}
%\ee
%this leads to $J_\e^{(1b)}+J_\e^{(2a)}\to0$ as $\e\to0$. Indeed, for $y$ in an $\e$-neighborhood of any $Y(\al,0)$, i.e. not necessarily $\al\in\Gamma$, we have

Combining \eqref{jlst1a-v4}, \eqref{jlst-pb-lim}, \eqref{jlst-pc-lim}, and \eqref{jlst-2a-st2} gives the result, which, in compact form, can be written as follows
\be\label{f1lim}\begin{split}
&\lim_{\e\to0}f_{\chi\e}(x_0)\\
&=\int_{S_+^2} B(\al) [G_0(\pa_t P)^2\de(P_0)+(G_0P'_yY''_{tt}+\pa_t^2((G-G_0)L)+G_0P''_{yy}Y'_t\cdot Y'_t)\theta(P_0)] d\al\\
&=\int_{S_+^2} B(\al) [G_0(\pa_t P)^2\de(P_0)+(G_0\pa_t^2P+\pa_t^2((G-G_0)L))\theta(P_0)] d\al.
\end{split}
\ee
Here $G_0:=G(Y(\al,0))$, $G:=G(Y(\al,t))$, $P:=P(Y(\al,t))$, $P_0:=P(Y(\al,0))$, and the derivatives with respect to $t$ are evaluated at $t=0$. This coincides with what we get by substituting $g=P_+G$ into the continuous inversion formula \eqref{inv-formula}. Indeed, representing $P_+G=(GP)\theta(P)$, we have 
\be\label{cont-ver}\begin{split}
\pa_t^2((GP)\theta(P))&=\pa_t[\pa_t(GP)\theta(P)+(GP)\de(P)\pa_t P]=\pa_t[\pa_t(GP)\theta(P)]\\
&=\pa_t(GP)\pa_t P\de(P_0)+\pa_t^2(GP)\theta(P_0)\\
&=G_0(\pa_t P)^2\de(P_0)+\pa_t^2(GP)\theta(P_0).
\end{split}
\ee
The coefficients in front of the delta-function in \eqref{f1lim} and \eqref{cont-ver} match. Subtracting the coefficients in front of the Heaviside function gives:
\be\label{coef-diff}\begin{split}
[G_0\pa_t^2P+\pa_t^2((G-G_0)L)] - \pa_t^2(GP)=-\pa_t^2[(G-G_0)(P-L)]=0.
\end{split}
\ee
Here we have used that $G_0$ is independent of $t$, and the expression under the derivative has a zero of third order at $t=0$. Thus the theorem is proven in the case $x=x_0$.

Next, consider the case of a general $\xh:=x_0+\e \tilde x$ (cf. \eqref{res-remote}). We begin by repeating the steps \eqref{GRT-loc-v2-parts}--\eqref{jlst-v2}, where all the auxiliary functions, such as $Y$, are computed using $\xh$ instead of $x_0$. It is clear that in any place where an auxiliary function is not divided by $\e$, e.g. $B(\al)$ in \eqref{inv-formula-p2} and $e(\al)$, $u(\al)$ in \eqref{psi-new}, replacing $\xh$ with $x_0$ introduces an error of magnitude $O(\e)$. Here we also used the property (2) of $\psi$. Consequently, the analogue of \eqref{jlst-v2} for $\xh$ becomes:
\be\label{jlst-alt}\begin{split}
J_\e^{(1a)}(\xh)=&\frac1{\e}\int_{S_+^2}B_1(\al) \psi\left(\frac{P(Y(\al,0;\xh))}{\e},\frac{Y(\al,0;\xh)}\e;\al\right) d\al + O(\e),
\end{split}
\ee
where only $Y$ is different from the analogous function in \eqref{jlst-v2}. Note that $P(y)$ depends only on the shape of $\s$ in a neighborhood of $x_0$ and, therefore, is independent of $\xh$. We have
\be\label{t0Y-expansion}\begin{split}
Y(\al,0;\xh)&=Y(\al,0)+\e W(\al,\tilde x) + O(\e^2),\\ 
P(Y(\al,0;\xh))&=P(Y(\al,0))+\e P'_y(Y(\al,0))W(\al,\tilde x)+O(\e^2),
\end{split}
\ee
for some smooth and bounded $W$. Here $Y(\al,0)$ is the same as in \eqref{jlst-v2}. Substituting into \eqref{jlst-alt} gives
\be\label{jlst-alt-smpl}\begin{split}
J_\e^{(1a)}(\xh)=&\frac1{\e}\int_{S_+^2}B_1(\al) \psi\biggl(\frac{P(Y(\al,0))}{\e}+P'_y(Y(\al,0))W(\al,\tilde x),\\
&\hspace{3cm}\frac{Y(\al,0)}\e+W(\al,\tilde x);\al\biggr) d\al + O(\e).
\end{split}
\ee
Similarly to \eqref{psi1def}, introduce
\be\label{shifted-psi}\begin{split}
\psi_2(s,y)&:=B_1(\al)\psi(\tilde s_1+P'_y(Y(\al,0))W(\al,\tilde x),y+W(\al,\tilde x);\al),\\
\al&=\al(0,s_2),s=(\tilde s_1,s_2).
\end{split}
\ee
The point $\tilde x$ is fixed, so we do not need to list it in the arguments of $\psi_2$. Clearly, $\psi_2$ satisfies the same properties (1)--(3) as $\psi$. Hence Lemma~\eqref{key-ass-res} applies to $\psi_2$ as well, and we get similarly to \eqref{jlst-v3}, \eqref{psi1def}, and \eqref{jlst1a-v2}:
\be\label{jlst1a-v2-alt}\begin{split}
\lim_{\e\to0}J_\e^{(1a)}(\xh)&=
\int_{\br^2} \left[\int_{[0,1]^3}\psi_2\left(\tilde s_1,v;\al(0,s_2)\right)dv\right] d\tilde s_1 ds_2\\
&=
\int_{\br^2} B_2(s_2)\left[\int_{[0,1]^3}\psi\left(\tilde s_1,v;\al(0,s_2)\right)dv\right] d\tilde s_1 ds_2.
\end{split}
\ee
Here we have used that the integrals with respect to $\tilde s_1$ and $v$ are unaffected by the constant (with respect to $\tilde s_1$ and $v$) shifts in \eqref{shifted-psi}. Therefore, \eqref{jlst1a-v4} holds with $J_\e^{(1a)}(\xh)$ on the left.

To find the limit of $J_\e^{(1b)}(\xh)$, consider the key step in \eqref{jlst-pb-lim}:
\be\label{jlst-pb-xh}
\sum_j L_+(\e j;\xh) \ik\left(\frac{Y(\al,0;\xh)+y-\e j}\e\right)
=L_+(Y(\al,0;\xh)+y;\xh),
\ee
which is rewritten with $x_0$ replaced by $\xh$. This equality holds everywhere except in an $O(\e)$ neighborhood of $\Gamma$ ($=\Gamma_{x_0}$). Here we use that the curve $\Gamma_{\xh}$, which is obtained by solving $P(Y(\al,0;\xh))=0$, depends smoothly on $\xh$, and $\text{dist}(\Gamma_{\xh},\Gamma)=O(\e)$ (see the argument preceding the statement of Theorem~\ref{thm:remote}). Similarly to \eqref{jlst-pb-lim}, the integrand is uniformly bounded, and we get
\be\label{jlst-pb-lim-xh}
\lim_{\e\to0}J_\e^{(1b)}(\xh)=\int_{S_+^2}B_1(\al) P'_y(Y(\al,0))\cdot Y''_{tt}(\al,0) \theta(P(Y(\al,0)))d\al.
\ee
The fact that the limits of $J_\e^{(1c)}(\xh)$ and $J_\e^{(2)}(\xh)$ as $\e\to0$ are independent of $\tilde x\in\tilde\us$ can be established in a similar way, and the theorem is proven.
\end{proof}

\section{Numerical experiment}\label{numerix}

We start by constructing an interpolation kernel with the required properties. To obtain $\varphi$, we first obtain an interpolation kernel $\varphi_{\text{1D}}$ that has properties IK1--IK5 in $\br$, and then extend it to $\br^3$ in a separable fashion. To obtain $\varphi_{\text{1D}}$ we use the result of \cite{kat19a}, where such a kernel is obtained following the method in \cite{btu03}:
\be\label{final-form}
\varphi_{\text{1D}}(t)=0.5(B_3(t)+B_3(t-2))+4B_3(t-1)-2(B_4(t)+B_4(t-1)).
\ee  
Here $B_n$ is the cardinal B-spline of degree $n$ supported on $[0,n+1]$. 
Then the kernel $\varphi$ becomes
\be\label{vp-form}
\varphi(y)=\prod_{k=1}^3\varphi_{\text{1D}}\left(\frac{y_k}{\Delta_k}+3\right),\ y=(y_1,y_2,y_3),
\ee  
where $\Delta_k$ is the data stepsize along the $k$-th axis. For simplicity, in this paper all the $\Delta_k$ are equal, i.e. $\Delta_k=\e$, $k=1,2,3$. 

The GRT we consider here integrates a function supported in the half-space $x_3>0$ over spheres that are tangent to the plane $x_3=0$. The family of such spheres is three-dimensional. We parametrize the spheres (and, consequently, the GRT) by the coordinates of their center $y$. Thus, the surfaces $\s_y$ are spheres, and the defining function $\Phi$ in \eqref{grt_1} becomes:
\be\label{Phi-ex}
\Phi(x,y):=y_3^2-(x_1-y_1)^2-(x_2-y_2)^2-(x_3-y_3)^2=0.
\ee  
Clearly,
\be\label{Phi-ders-ex}
\Phi'_x(x,y)=2(y-x),\ \Phi'_y(x,y)=2(x_1-y_1,x_2-y_2,x_3).
\ee  

The test object is the ball with center $x_c=(0,0,11)$, radius $R=5$, and uniform density 1. The point on the boundary $x_0$, in a neighborhood of which we compute resolution, is given by
\be\label{point}
x_0=x_c-R\al_0,\ \al_0=(\sin(0.2\pi)\cos(0.7\pi),\sin(0.2\pi)\sin(0.7\pi),\cos(0.2\pi)).
\ee
In agreement with our convention, $\al_0$ points into the interior of the ball.

There can be two spheres that are tangent to the ball at $x_0$. As an example, we consider the sphere whose center $y_0$ satisfies $(x_0-y_0)\cdot\al_0>0$. Thus, for reconstruction near $x_0$ we use the data in a neighborhood of $y_0$. With this choice of $y_0$, the condition $\Phi'_x(x_0,y_0)/|\Phi'_x(x_0,y_0)|=-\al_0$ (see the text following \eqref{rotate-xy}) is satisfied with $\Phi$ given by \eqref{Phi-ex}. For the selected $x_0$, $\al_0$, and $y_0$, we compute using 
\eqref{Phi-ders-ex}:
\be\label{chi-num}
\mc=|\Phi'_x|/|\Phi'_y|=0.526.
\ee

To compute the GRT, we use the formula for the area of the spherical cap:
\be\label{sph-cap}
A=	2\pi R h,
\ee
where $R$ is the radius of the sphere, and $h$ is the height of the cap. The values of $R$ and $h$ can be computed once the center of the sphere $\s_y$ is chosen (e.g., $R=y_3$). To simulate discrete data, the GRT is computed at the points $y=r+\e j$. The interpolated data $g_\e$ is computed using \eqref{f-int}, where the kernel is given by \eqref{vp-form} with $\Delta_k=\e$, $k=1,2,3$.

To apply the inversion formula \eqref{inv-formula}, we numerically integrate $g_\e$ over a neighborhood of $\al_0$ on the unit sphere. To compute $(\pa/\pa t)^2 g_\e(Y(\al,t;x))$ at $t=0$ we use \eqref{f-int} and the chain rule as in \eqref{ikderive}. Given $x$, $\al$, and $t$, the center of the sphere containing the point $x+t\al$ and normal to $\al$ at that point (cf. the paragraph following \eqref{f-int}) is easily found to be:
\be\label{y-sol}
Y(\al,t;x)=(x+t\al)-\frac{x_3+t\al_3}{1+\al_3}\al.
\ee
Consequently,
\be\label{y-sol-ders}
\left.\frac{\pa}{\pa t}Y(\al,t;x)\right|_{t=0}=\frac{1}{1+\al_3}\al,\ \left.\frac{\pa^2}{\pa t^2}Y(\al,t;x)\right|_{t=0}=0,
\ee
and 
\be\label{ikderive-2}\begin{split}
&\left.\left(\frac{\pa}{\pa t}\right)^2 \ik\left(\frac{Y(\al,t;x)-\e j}\e\right)\right|_{t=0} d\al
=\frac1{\e^2}\frac{1}{(1+\al_3)^2}\sum_{i,k=1}^3\ik''_{ik}\left(\frac{Y(\al,0;x)}\e-j\right)\al_i\al_k.
\end{split}
\ee

The cut-off function $\chi$ in \eqref{inv-formula} is constructed as follows. Let $\al^\perp$ run through the unit sphere in the plane $\al_0^\perp$. Then any $\al\in S_+^2$ ($\al_0$ is the North pole of $S_+^2$) can be represented in the form $\al=(\cos\om) \al_0+(\sin\om)\al^\perp$, $0\le\om\le \pi/2$. In the code we use
\be\label{cutoff}
\chi(\al)=\begin{cases}1,& 0\le\om < 0.8\om_{\text{mx}},\\
\frac{1+\cos((\om-0.8\om_{\text{mx}})/(0.2\om_{\text{mx}}))}2,& 0.8\om_{\text{mx}}\le\om < \om_{\text{mx}},\\
0,&\om \ge \om_{\text{mx}}.
\end{cases}
\ee

Finally, the predicted response is computed using \eqref{main-res}. The results corresponding to $\e=0.01$ are shown in Figure~\ref{fig:01}. We see a good match between the predicted and actual transition curves.

%The results corresponding to $\e=0.02$, $\e=0.015$, and $\e=0.01$ are shown in Figures~\ref{fig:02}, \ref{fig:015}, and \ref{fig:01}, respectively. Our goal is to study how singularities are recovered, so the reconstructed curve is shifted by a constant value to match the predicted behavior at $h=0$. The $x$-axes in these figures show the $h$-values. The $y$-axes show the (shifted) reconstructed values at the points $\xh=x_0+\e\al_0h$ (cf. \eqref{rec-pt}) as well as the predicted transition curve. We see that the predicted and actual transition curves become quite close when $\e$ goes down from $0.02$ to $0.015$ (see Figures~\ref{fig:02} and \ref{fig:015}). The difference between the two fails to become even smaller when $\e$ goes down to $0.015$ (see Figures~\ref{fig:02} and \ref{fig:015}) because of the onset of numerical instabilities.

\begin{figure}[h]
{\centerline{\epsfig{file=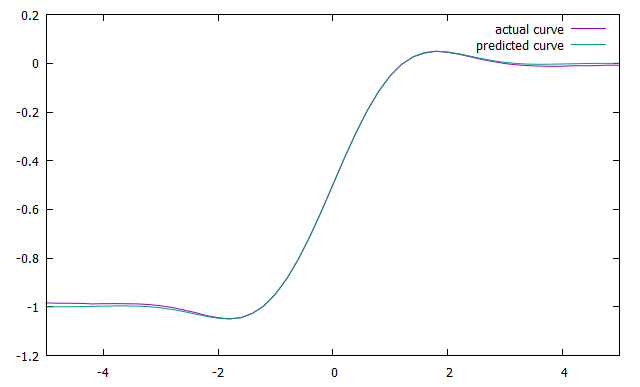}.png, width=10cm}}
}
\caption{Comparison of the predicted and actual transition curves for $\e=0.01$.}
\label{fig:01}
\end{figure}

\bibliographystyle{plain}
\bibliography{bibliogr_A-K,bibliogr_L-Z}

\begin{thebibliography}{10}

\bibitem{ahw12}
F.~Andersson, M.~V.~De Hoop, and H.~Wendt.
\newblock {Multiscale Discrete Approximation of Fourier Integral Operators}.
\newblock {\em Multiscale Modeling and Simulation}, 10:111--145, 2012.

\bibitem{belk}
G.~Beylkin.
\newblock The inversion problem and applications of the generalized {Radon}
  transform.
\newblock {\em Comm. Pure and Appl. Math.}, 37:579--599, 1984.

\bibitem{btu03}
T.~Blu, P.~Th{\'e}venaz, and M.~Unser.
\newblock {Complete Parameterization of Piecewise-Polynomial Interpolation
  Kernels}.
\newblock {\em IEEE Transactions on Image Processing}, 12:1297--1309, 2003.

\bibitem{cdy07}
E.~Candes, L.~Demanet, and L.~Ying.
\newblock Fast computation of {Fourier} integral operators.
\newblock {\em SIAM Journal on Scientific Computing}, 29:2464--2493, 2007.

\bibitem{cdy09}
E.~Candes, L.~Demanet, and L.~Ying.
\newblock A fast buttery algorithm for the computation of {Fourier} integral
  operators.
\newblock {\em SIAM Multiscale Modeling and Simulation}, 7:1727--1750, 2009.

\bibitem{cb15}
M.~Cheney and B.~Borden.
\newblock {Synthetic Aperture Radar Imaging}.
\newblock In O.~Scherzer, editor, {\em Handbook of Mathematical Methods in
  Imaging}, pages 763--799. Springer, New York, NY, 2015.

\bibitem{gkqr18}
C.~Grathwohl, P.~Kunstmann, E.~T. Quinto, and A.~Rieder.
\newblock {Approximate inverse for the common offset acquisition geometry in 2D
  seismic imaging}.
\newblock {\em Inverse Problems}, 34, 2018.
\newblock article id 014002.

\bibitem{hor}
L.~Hormander.
\newblock {\em The Analysis of Linear Partial Differential Operators, {Vol I}}.
\newblock Springer Verlag, New York, 1983.

\bibitem{kat10b}
A.~Katsevich.
\newblock An accurate approximate algorithm for motion compensation in
  two-dimensional tomography.
\newblock {\em Inverse Problems}, 26, 2010.
\newblock article ID 065007 (16 pp).

\bibitem{kat_2017}
A.~Katsevich.
\newblock A local approach to resolution analysis of image reconstruction in
  tomography.
\newblock {\em SIAM Journal on Applied Mathematics}, 77:1706--1732, 2017.

\bibitem{kat19a}
A.~Katsevich.
\newblock Analysis of reconstruction from discrete radon transform data in
  $\mathbb r^3$ when the function has jump discontinuities.
\newblock {\em SIAM Journal on Applied Mathematics}, 2019.
\newblock to appear.

\bibitem{kk15}
P.~Kuchment and L.~Kunyansky.
\newblock {Mathematics of Photoacoustic and Thermoacoustic Tomography}.
\newblock In O.~Scherzer, editor, {\em Handbook of Mathematical Methods in
  Imaging}, pages 1117--1167. Springer, New York, NY, 2015.

\bibitem{KN_06}
L.~Kuipers and H.~Niederreiter.
\newblock {\em Uniform Distribution of Sequences}.
\newblock Dover Publications, Inc., Mineola, NY, 2006.

\bibitem{nat3}
F.~Natterer.
\newblock {\em {The Mathematics of Computerized Tomography}}.
\newblock SIAM, Philadelphia, 2001.

\bibitem{qrs11}
E.~T. Quinto, A.~Rieder, and Th. Schuster.
\newblock Local inversion of the sonar transform regularized by the approximate
  inverse.
\newblock {\em Inverse Problems}, 27, 2011.
\newblock article id 035006.

\bibitem{rz2}
A.G. Ramm and A.I. Zaslavsky.
\newblock Reconstructing singularities of a function given its {Radon}
  transform.
\newblock {\em Math. and Comput. Modelling}, 18(1):109--138, 1993.

\bibitem{rz1}
A.G. Ramm and A.I. Zaslavsky.
\newblock Singularities of the {Radon} transform.
\newblock {\em Bull. Amer. Math. Soc.}, 25:109--115, 1993.

\bibitem{stef18}
P.~Stefanov.
\newblock Semiclassical sampling and discretization of certain linear inverse
  problems.
\newblock arXiv:1811.01240, 2018.

\bibitem{tkk18}
F.~Terzioglu, P.~Kuchment, and L.~Kunyansky.
\newblock Compton camera imaging and the cone transform: a brief overview.
\newblock {\em Inverse Problems}, 34, 2018.
\newblock article id 054002.

\bibitem{trev2}
F.~Treves.
\newblock {\em {Introduction to Pseudodifferential and Fourier Integral
  Operators. Volume 2: Fourier Integral Operators}}.
\newblock The University Series in Mathematics. Plenum, New York, 1980.

\bibitem{wa15}
K.~Wang and M.~A. Anastasio.
\newblock {Photoacoustic and Thermoacoustic Tomography: Image Formation
  Principles}.
\newblock In O.~Scherzer, editor, {\em Handbook of Mathematical Methods in
  Imaging}, pages 1081--1116. Springer, New York, NY, 2015.

\bibitem{yang15}
H.~Yang.
\newblock {\em {Oscillatory Data Analysis and Fast Algorithms for Integral
  Operators}}.
\newblock PhD thesis, Stanford University, 2015.

\bibitem{zwor12}
M.~Zworski.
\newblock {\em Semiclassical analysis}, volume 138 of {\em Graduate Studies in
  Mathematics}.
\newblock American Mathematical Society, Providence, RI, 2012.

\end{thebibliography}
\end{document}